\documentclass[a4paper,11pt,final]{article}

\usepackage{graphicx}
\usepackage{amssymb,amsmath,amsthm,mathrsfs,xfrac}
\usepackage{enumerate}
\usepackage{a4wide}
\usepackage{hyperref}
\usepackage{array}
\usepackage{color}

\numberwithin{equation}{section}
\allowdisplaybreaks[4]

\newtheorem{proposition}{Proposition}[section]
\newtheorem{theorem}[proposition]{Theorem}
\newtheorem{lemma}[proposition]{Lemma}
\newtheorem{definition}[proposition]{Definition}

\newtheorem{remark}[proposition]{Remark}

\renewenvironment{proof}{\smallskip\noindent\emph{\textbf{Proof.}}%
  \hspace{1pt}}{\hspace{-5pt}{\nobreak\quad\nobreak\hfill\nobreak%
    $\square$\vspace{2pt}\par}\smallskip\goodbreak}

\newenvironment{proofof}[1]{\smallskip\noindent{\textbf{Proof~of~#1.}}%
  \hspace{1pt}}{\hspace{-5pt}{\nobreak\quad\nobreak\hfill\nobreak%
    $\square$\vspace{2pt}\par}\smallskip\goodbreak}

\newcommand{\pint}[1]{\mathaccent23{#1}}

\newcommand{\C}[1]{\mathbf{C}^{#1}}
\newcommand{\Cc}[1]{\mathbf{C}_c^{#1}}

\newcommand{\BV}{\mathbf{BV}}
\newcommand{\PC}{\mathbf{PC}}
\newcommand{\PLC}{\mathbf{PLC}}
\renewcommand{\L}[1]{{\mathbf{L}^#1}}

\newcommand{\Wloc}[2]{{\mathbf{W}_{\mathbf{loc}}^{#1,#2}}}
\newcommand{\modulo}[1]{{\left|#1\right|}}
\newcommand{\norma}[1]{{\left\|#1\right\|}}
\newcommand{\caratt}[1]{{\chi_{\strut#1}}}
\newcommand{\reali}{{\mathbb{R}}}
\newcommand{\naturali}{{\mathbb{N}}}
\newcommand{\interi}{{\mathbb{Z}}}
\renewcommand{\epsilon}{\varepsilon}
\renewcommand{\phi}{\varphi}
\renewcommand{\theta}{\vartheta}
\renewcommand{\O}{\mathcal{O}(1)}
\newcommand{\tv}{\mathinner{\rm TV}}

\newcommand{\sgn}{\mathop{\rm sgn}}
\newcommand{\sgnp}{\operatorname{sgn}^{+}}
\newcommand{\sgnm}{\operatorname{sgn}^{-}}

\newcommand{\Lip}{\mathinner\mathbf{Lip}}
\renewcommand{\d}[1]{\mathinner{\mathrm{d}{#1}}}

\DeclareMathOperator*{\esslim}{ess\,lim}
\DeclareMathOperator*{\essinf}{ess\,inf}
\DeclareMathOperator*{\esssup}{ess\,sup}

\newcommand{\p}{\mathbf{p}}
\newcommand{\ub}{\boldsymbol{u_b}}

\title{Stability of the 1D IBVP for a\\ Non Autonomous Scalar Conservation Law}

\author{Rinaldo M.~Colombo\footnote{\texttt{rinaldo.colombo@unibs.it}}
  \qquad Elena Rossi\footnote{\texttt{elena.rossi@unibs.it}} \\ INDAM
  Unit, University of Brescia, Italy}
\date{}

\begin{document}

\maketitle

\begin{abstract}

  \noindent We prove the stability with respect to the flux of
  solutions to initial -- boundary value problems for scalar
  non autonomous conservation laws in one space dimension. Key
  estimates are obtained through a careful construction of the
  solutions.

  \medskip

  \noindent\textit{2010~Mathematics Subject Classification:} 35L65,
  35L04

  \medskip

\end{abstract}

% \tableofcontents

\section{Introduction}
\label{sec:Intro}

This paper deals with the Initial Boundary Value Problem (IBVP) for a
possibly non autonomous scalar conservation law on a half--line
\begin{equation}
  \label{eq:1}
  \left\{
    \begin{array}{l@{\qquad}r@{\,}c@{\,}l}
      \partial_t u + \partial_x f (t,u) = 0
      & (t,x)
      & \in
      & [0,T] \times \reali_+
      \\
      u (0,x) = u_o (x)
      & x
      & \in
      & \reali_+
      \\
      u (t, 0) = u_b (t)
      & t
      & \in
      & [0,T] \,,
    \end{array}
  \right.
\end{equation}
or on a segment
\begin{equation}
  \label{eq:1sg}
  \left\{
    \begin{array}{l@{\qquad}r@{\,}c@{\,}l}
      \partial_t u + \partial_x f (t,u) = 0
      & (t,x)
      & \in
      & [0,T] \times [0,L]
      \\
      u (0,x) = u_o (x)
      & x
      & \in
      & [0,L]
      \\
      u (t, 0) = u_{b,1} (t)
      & t
      & \in
      & [0,T]
      \\
      u (t,L) = u_{b,2} (t)
      & t
      & \in
      & [0,T] \,.
    \end{array}
  \right.
\end{equation}
For these problems, we complete the basic well posedness and
stability results. That is, we detail below the proofs of the
existence of solutions and of their stability with respect to the
flow. For the Lipschitz continuous dependence of solutions on initial
and boundary data we refer to~\cite{BardosLerouxNedelec, bordo}.

With a slight abuse of notation, we refer to the non autonomous
(time dependent), respectively autonomous (time independent), case
as to the case where the flux $f$ depends explicitly on time $t$ or
not. In both cases, boundary data are time dependent.

Conservation Laws are typically studied either in the case of one
dimensional systems or of scalar multi--dimensional equations. In the
former case, we refer to~\cite{AmadoriColombo97, DuboisLeFloch,
  Goodman} for the basic existence results and for discussions on the
very definition of solution to the initial boundary value
problem. Differently from these works, the present paper deals with
the stability with respect to the flow and covers also the case of a
time dependent flow.

In the scalar multi--dimensional case, the key reference
is~\cite{BardosLerouxNedelec}, see also~\cite{bordo, DafermosBook,
  MalekEtAlBook, Martin, OttoPhD, OttoCR, Serre2, Vovelle}, which
considers the existence of solutions and their continuous dependence
on initial and boundary data but only on bounded domains. Here, in
addition, we deal also with unbounded domains and ensure the stability
with respect to the flow, though limited to the one dimensional
case. We stress here the key role played by the definition of
solutions to~\eqref{eq:1} or~\eqref{eq:1sg} as provided
in~\cite{Martin, Vovelle}. Indeed, this definition is stable under
$\L1$--convergence, see~\cite[Chapter~2, Remark~7.33]{MalekEtAlBook},
and its use allows to avoid all issues related to the limit of traces
converging to the trace of the limit.

Recall that in the case of the autonomous Cauchy problem, the
stability of solutions with respect to the flux is treated
in~\cite[Theorem~2.13]{HoldenRisebro}. In one space dimension,
\cite[Theorem~2.6]{BianchiniColombo} deals with a convex scalar
time independent flux, while autonomous systems are considered
in~\cite[Theorem~2.1]{BianchiniColombo}. Here, we extend these results
to the non autonomous case with boundary, albeit in the scalar one
dimensional case.

A key role in this paper is played by the \emph{wave front tracking}
technique, see~\cite{BressanLectureNotes, DafermosBook}. In this
framework, Glimm type functionals yield a precise control of the total
variation. As a consequence, we obtain the stability of solutions with
respect to the flux in the autonomous case, thanks to a careful use
of~\cite[Theorem~2.9]{BressanLectureNotes}. All these estimates then
lead to the stability in the time dependent case.

The next section presents the results concerning~\eqref{eq:1} on a
half line. Section~\ref{sec:SG} deals with~\eqref{eq:1sg} on a
segment. In both cases, we present first the autonomous case and then
the non autonomous one. Section~\ref{sec:Tech} is devoted to proofs.

\section{The Case of the Half--Line}
\label{sec:HalfLine}

All statements and proofs below are referred to the time interval
$[0,T]$ for a fixed $T>0$. Where the extension to $t \in \reali_+$ is
not straightforward, we provide all necessary details. Denote
$\reali_+ = \left[0, +\infty\right[$ and
$\pint{\reali}_+ = \left]0, +\infty\right[$.  Following~\cite{Martin,
  Vovelle}, for $a,b \in \reali$, we let
\begin{equation}
  \label{eq:5}
  \mathcal{I} (a,b) = \left[ \min\{a,b\}, \, \max\{a,b\}\right] \,.
\end{equation}
Below, if $u_\ell \in \L\infty (I_\ell;\reali)$ for real intervals
$I_\ell$ and for $\ell=1, \ldots, m$, we define
\begin{equation}
  \label{eq:19}
  \mathcal{U} (u_1, \ldots, u_m)
  =
  \bigl[
  \min_{\ell=1, \ldots, m} \essinf_{I_\ell} u_\ell \,, \;
  \max_{\ell=1, \ldots, m} \esssup_{I_\ell} u_\ell
  \bigr] \,.
\end{equation}
Equivalently, $\mathcal{U} (u_1, \ldots, u_m)$ is the closed convex
hull of $\bigcup_{\ell=1}^m u_\ell (I_\ell)$. If $I_u$ is a real
interval, for $u \in \BV (I_u; \reali)$, $\tv (u)$ stands for the
total variation of $u$ on $I_u$, see~\cite[\S~5.10.1]{EvansGariepy}
and, for any interval $I \subseteq I_u$, we also set
$\tv (u;I) = \tv (u_{|I})$.  Moreover, for
$\boldsymbol{u} \in \BV (I; \reali^m)$, we define
$ \tv(\boldsymbol{u}) = \sum_{\ell=1}^m \tv (u_\ell)$.
Denote by ${\mathscr{T}}_t$ the $t$--translation operator:
\begin{equation}
  \label{eq:Trasl}
  \left({\mathscr{T}}_t u\right) (\tau) = u (t+\tau).
\end{equation}
As usual, $u (t,0+)= \displaystyle\lim_{x \to 0+} u(t,x)$ stands for
the trace at $0$ from the right,
see~\cite[Paragraph~5.3]{EvansGariepy} or~\cite[Appendix]{bordo}.
Throughout, we set
\begin{equation}
  \label{eq:sgn}
  \begin{array}{rcl@{\qquad\qquad}rcl}
    \sgnp (u)
    & =
    & \begin{cases}
      1 & \mbox{if } u>0,
      \\
      0 & \mbox{if } u \leq 0,
    \end{cases}
    & \sgnm (u)
        & =
        & \begin{cases}
          0 & \mbox{if } u \geq 0,
          \\
          -1 & \mbox{if } u < 0,
        \end{cases}
    \\
    u^+
    & =
    & \max\{u, 0\} ,
        & u^-
        & =
    & \max\{-u, 0\} .
  \end{array}
\end{equation}
Introduce the \emph{semi--Kru\v{z}kov entropy--entropy flux pairs},
see~\cite{Martin, Vovelle}: for any $k \in \reali$
\begin{equation}
  \label{eq:sken}
  \begin{aligned}
    \eta_k^+ (u) = \ & \left(u-k\right)^+, & & & \Phi_k^+ (t,u) = \ &
    \sgnp (u-k) \; \left(f (t,u) - f (t,k)\right),
    \\
    \eta_k^- (u) = \ & \left(u-k\right)^-, & & & \Phi_k^- (t,u) = \ &
    \sgnm (u-k) \; \left(f (t,u) - f (t,k)\right).
  \end{aligned}
\end{equation}

\begin{definition}
  \label{def:solsk}
  A \emph{weak entropy solution} to the IBVP~\eqref{eq:1} on the
  interval $[0,T]$ is a map
  $u \in \L\infty\left([0,T] \times \reali_+; \reali \right)$ such
  that for any $k \in \reali$ and for any test function
  $\phi \in \Cc1 (\reali \times \reali; \reali_+)$
  \begin{equation}
    \label{eq:ei1}
    \begin{aligned}
      & \int_0^T \int_{\reali_+} \left\{ \eta_k^+ \left(u (t,x)\right)
        \, \partial_t \phi (t,x) + \Phi_k^+\left(t,u (t,x)\right)
        \, \partial_x \phi (t,x) \right\} \d{x} \d{t}
      \\
      & + \int_{\reali_+} \eta_k^+ \left(u_o (x)\right) \, \phi (0,x)
      \d{x} - \int_{\reali_+} \eta_k^+ \left(u (T,x)\right) \, \phi
      (T,x) \d{x}
      \\
      & + \norma{\partial_u f}_{\L\infty ([0,T] \times
        \mathcal{U};\reali)} \int_0^T \eta_k^+ \left(u_b (t)\right) \,
      \phi (t,0) \d{t} \geq 0,
    \end{aligned}
  \end{equation}
  and
  \begin{equation}
    \label{eq:ei2}
    \begin{aligned}
      & \int_0^T \int_{\reali_+} \left\{ \eta_k^- \left(u (t,x)\right)
        \, \partial_t \phi (t,x) + \Phi_k^-\left(t,u (t,x)\right)
        \, \partial_x \phi (t,x) \right\} \d{x} \d{t}
      \\
      & + \int_{\reali_+} \eta_k^- \left(u_o (x)\right) \, \phi (0,x)
      \d{x} - \int_{\reali_+} \eta_k^- \left(u (T,x)\right) \, \phi
      (T,x) \d{x}
      \\
      & + \norma{\partial_u f}_{\L\infty ([0,T] \times
        \mathcal{U};\reali)} \int_0^T \eta_k^- \left(u_b (t)\right) \,
      \phi (t,0) \d{t} \geq 0 ,
    \end{aligned}
  \end{equation}
  where $\mathcal{U} = \mathcal{U} (u_o, {u_b}_{|[0,T]})$ as
  in~\eqref{eq:19}.
\end{definition}

Relying essentially solely on Definition~\ref{def:solsk}, one obtains
the Lipschitz continuous dependence of the solution
to~\eqref{eq:1} on initial and boundary data.

\begin{proposition}
  \label{prop:lipdep}
  Let $f \in \C1([0,T] \times \reali;\reali)$ be such that
  $\{u \mapsto \partial_t f(t.u)\} \in \Wloc{1}{\infty}
  (\reali;\reali)$
  for all $t \in [0,T]$,
  $u_o,w_o \in (\L1 \cap \L\infty)(\reali_+; \reali)$ and
  $u_b, w_b \in (\L1 \cap \L\infty)([0,T]; \reali)$. Assume the
  problems
  \begin{displaymath}
    \!\!\!\!\!\!\!\!
    \left\{
      \begin{array}{lr@{\,}c@{\,}l@{}}
        \partial_t u + \partial_x f (t,u) = 0
        & (t,x)
        & \in
        & [0,T] \times \reali_+
        \\
        u (0,x) = u_o (x)
        & x
        & \in
        & \reali_+
        \\
        u (t, 0+) = u_b (t)
        & t
        & \in
        & [0,T]
      \end{array}
    \right.
    \mbox{and }
    \left\{
      \begin{array}{lr@{\,}c@{\,}l@{}}
        \partial_t w + \partial_x f (t,w) = 0
        & (t,x)
        & \in
        & [0,T] \times \reali_+
        \\
        w (0,x) = w_o (x)
        & x
        & \in
        & \reali_+
        \\
        w (t, 0+) = w_b (t)
        & t
        & \in
        & [0,T]
      \end{array}
    \right.
    \!\!
  \end{displaymath}
  admit solutions $u,w \in \L\infty ([0,T] \times \reali_+; \reali)$
  in the sense of Definition~\ref{def:solsk}, such that $u$ and $w$
  both admit a trace for $x \to 0+$ for a.e.~$t \in [0,T]$. Then, for
  all $t \in [0,T]$,
  \begin{displaymath}
    \norma{u (t) - w (t)}_{\L1 (\reali_+; \reali)}
    \leq
    \norma{u_o - w_o}_{\L1 (\reali_+; \reali)}
    +
    \norma{\partial_u f}_{\L\infty ([0,t] \times \mathcal{U};\reali)} \;
    \norma{u_b - w_b}_{\L1 ([0,t]; \reali)}
  \end{displaymath}
  where
  $\mathcal{U} = \mathcal{U} ({u_b}_{|[0,t]}, \, {w_b}_{|[0,t]})$ is
  as in~\eqref{eq:19}.
\end{proposition}

Remark that Proposition~\ref{prop:lipdep}, whose proof is deferred to
Section~\ref{sec:Tech}, also ensures the uniqueness of the solution
to~\eqref{eq:1} in the sense of Definition~\ref{def:solsk}, as soon as
a solution exists.

\subsection{The Autonomous Case on the Half-Line}
\label{sec:autHL}

We study first the following autonomous IBVP, which is a particular case of~\eqref{eq:1}:
\begin{equation}
  \label{eq:1aut}
  \left\{
    \begin{array}{l@{\qquad}r@{\,}c@{\,}l}
      \partial_t u + \partial_x f (u) = 0
      & (t,x)
      & \in
      & [0,T] \times \reali_+
      \\
      u (0,x) = u_o (x)
      & x
      & \in
      & \reali_+
      \\
      u (t, 0) = u_b (t)
      & t
      & \in
      & [0,T] \, .
    \end{array}
  \right.
\end{equation}
Solutions to~\eqref{eq:1aut} are understood in the sense of
Definition~\ref{def:solsk}. Observe that Proposition~\ref{prop:lipdep}
applies to~\eqref{eq:1aut}, under the hypothesis
$f \in \C1 (\reali;\reali)$. The next Proposition ensures the
existence of solutions to~\eqref{eq:1aut}, as well as some of their
properties.

\begin{proposition}
  \label{prop:2}
  Let $f \in \Wloc{1}{\infty}(\reali; \reali)$,
  $u_o \in (\L1 \cap \BV) (\reali_+; \reali)$ and
  $u_b \in (\L1 \cap \BV) ([0,T]; \reali)$. Then,
  problem~\eqref{eq:1aut} admits a solution $u$ in the sense of
  Definition~\ref{def:solsk}, with the properties:
  \begin{enumerate}
  \item If $u_o$ and $u_b$ are piecewise constant, then for $t$ small,
    the map $t \to u (t)$ coincides with the gluing of Lax solutions
    to Riemann problems at the points of jumps of $u_o$ and at $x=0$.
  \item Range of $u$: with the notation in~\eqref{eq:19},
    $u (t,x) \in \mathcal{U} (u_o, {u_b}_{|[0,t]})$ for
    a.e.~$(t,x) \in [0,T] \times \reali_+$. Hence, for
    a.e.~$t \in [0,T]$,
    \begin{displaymath}
      \norma{u (t)}_{\L\infty
        (\reali_+; \reali)} \leq \max \left\{ \norma{u_o}_{\L\infty
          (\reali_+; \reali)}, \, \norma{u_b}_{\L\infty ([0,t];\reali)}
      \right\}.
    \end{displaymath}
  \item $u$ is Lipschitz continuous in time: for all
    $t_1,t_2 \in [0,T]$,
    \begin{displaymath}
      \norma{u (t_1) - u (t_2)}_{\L1 (\reali_+; \reali)}
      \leq
      C_1 \; \modulo{t_2 - t_1},
    \end{displaymath}
    where
    $C_1 = \norma{f'}_{\L\infty (\mathcal{U};\reali)} \left( \tv (u_o)
      + \tv (u_b; [0,t_1\vee t_2]) +\modulo{u_b (0+) - u_o (0+)}
    \right)$
    and $\mathcal{U} = \mathcal{U} (u_o, {u_b}_{|[0,t_1 \vee t_2]})$,
    with the notation~\eqref{eq:19}.
  \item Total variation estimate: for all $t \in [0,T]$
    \begin{displaymath}
      \tv\left(u (t) \right)
      \leq
      \tv (u_o)
      +
      \tv \left(u_b; [0,t]\right)
      +
      \modulo{u_b(0+) - u_o (0+)} \,.
    \end{displaymath}
  \end{enumerate}
\end{proposition}

\noindent The proof is deferred to \S~\ref{subs:AHL}. Various results
similar to, but not containing, Proposition~\ref{prop:2} can be found
in the current literature. The case of a convex flux is treated
in~\cite{AnconaMarson3}. A bounded domain is considered
in~\cite{BardosLerouxNedelec} and in~\cite{AmmarWittboldCarrillo,
  bordo}, see also~\cite[Section~6.9]{DafermosBook}
or~\cite[Section~15.1]{Serre2}.

Our main result, namely the stability of the solution
to~\eqref{eq:1aut} with respect to the flux, concludes this
Section. The current literature consider the case when no boundary is
present. In the one dimensional setting, the scalar equation is
treated in~\cite[Theorem~2.13]{HoldenRisebro}
and~\cite[Theorem~2.6]{BianchiniColombo} for a convex scalar flux,
while systems are considered
in~\cite[Theorem~2.1]{BianchiniColombo}. The multi-dimensional case is
covered in~\cite{ColomboMercierRosini}.

\begin{theorem}
  \label{thm:3}
  Let $f,g \in \C1 (\reali; \reali)$,
  $u_o \in (\L1 \cap \BV) (\reali_+; \reali)$ and
  $u_b \in (\L1 \cap \BV) ([0,T]; \reali)$. Call $u$ and $v$ the
  solutions to the problems
  \begin{equation}
    \label{eq:14}
    \!\!\!\!\!
    \left\{
      \begin{array}{@{\,}lr@{\,}c@{\,}l}
        \partial_t u + \partial_x f (u) = 0
        &(t,x)
        & \in
        & [0,T] \times \reali_+
        \\
        u (0,x) = u_o (x)
        & x
        & \in
        & \reali_+
        \\
        u (t, 0) = u_b (t)
        & t
        & \in
        & [0,T]
      \end{array}
    \right.
    \mbox{ and }
    \left\{
      \begin{array}{@{\,}lr@{\,}c@{\,}l}
        \partial_t v + \partial_x g (v) = 0
        & (t,x)
        & \in
        & [0,T] \times \reali_+
        \\
        v (0,x) = u_o (x)
        & x
        & \in
        & \reali_+
        \\
        v (t, 0) = u_b (t)
        & t
        & \in
        & [0,T]
      \end{array}
    \right.
    \!\!\!\!\!
  \end{equation}
  constructed in Proposition~\ref{prop:2}. Then, with
  $\mathcal{U} = \mathcal{U} (u_o, {u_b}_{|[0,t]})$ as
  in~\eqref{eq:19}, for all $t \in [0,T]$,
  \begin{displaymath}
    \begin{aligned}
      \norma{u (t) - v (t)}_{\L1 (\reali_+; \reali)} \leq \
      & \max \!
      \left\{1, \norma{g'}_{\L\infty (\mathcal{U}; \reali)}\right\} \,
      \norma{D(f - g)}_{\L\infty (\mathcal{U};\reali)}
      \\
      & \times \left(\tv (u_o) + \tv\left(u_b; [0,t]\right) +
        \modulo{u_b (0+) - u_o (0+)} \right) \, t.
    \end{aligned}
  \end{displaymath}
\end{theorem}

\noindent The proof is deferred to Section~\ref{subs:AHL}.

\subsection{The Non Autonomous Case on the Half-Line}
\label{sec:nonautHL}

The results obtained in Section~\ref{sec:autHL} are here extended to
problem~\eqref{eq:1}. We first generalize Proposition~\ref{prop:2}.

\begin{proposition}
  \label{prop:2t}
  Let $f$ be such that
  \begin{equation}
    \label{eq:f}
    f \in \C1 ([0,T] \times \reali; \reali)
    \quad \mbox{and} \quad
    \{t \mapsto \partial_u f(t,u)\} \in \Wloc{1}{\infty} ([0,T];\reali)
    \mbox{ for all } u \in \reali.
  \end{equation}
  Fix $u_o \in (\L1 \cap \BV) (\reali_+; \reali)$ and
  $u_b \in (\L1 \cap \BV) ([0,T]; \reali)$. Then,
  problem~\eqref{eq:1} admits a solution $u$ in the sense of
  Definition~\ref{def:solsk}, with the properties:
  \begin{enumerate}
  \item Range of $u$: with the notation in~\eqref{eq:19},
    $u (t,x) \in \mathcal{U} (u_o, {u_b}_{|[0,t]})$ for
    a.e.~$(t,x) \in [0,T] \times \reali_+$. Hence, for a.e.~$t\in [0,T]$,
    \begin{displaymath}
    \norma{u (t)}_{\L\infty (\reali_+; \reali)}
    \leq
    \max \left\{
      \norma{u_o}_{\L\infty (\reali_+; \reali)}, \,
      \norma{u_b}_{\L\infty ([0,t];\reali)}
    \right\}.
  \end{displaymath}
  \item $u$ is Lipschitz continuous in time: for all
    $t_1,t_2 \in [0,T]$,
    \begin{displaymath}
      \norma{u (t_1) - u (t_2)}_{\L1 (\reali_+; \reali)}
      \leq
      C \; \modulo{t_2 - t_1},
    \end{displaymath}
    where
    $C = \norma{\partial_u f}_{\L\infty ([0,t_1 \vee t_2] \times
      \mathcal{U};\reali)} \left( \tv (u_o) + \tv (u_b; [0,t_1\vee
      t_2]) +\modulo{u_b (0+) - u_o (0+)} \right)$
    and $\mathcal{U} = \mathcal{U} (u_o, {u_b}_{|[0,t_1 \vee t_2]})$,
    with the notation~\eqref{eq:19}.
  \item Total variation estimate: for all $t \in [0,T]$
    \begin{displaymath}
      \tv\left(u (t) \right)
      \leq
      \tv (u_o)
      +
      \tv \left(u_b; [0,t]\right)
      +
      \modulo{u_b(0+) - u_o (0+)} \,.
    \end{displaymath}
  \end{enumerate}
\end{proposition}

\noindent The proof is deferred to~\S~\ref{subs:NAHL}.

\begin{theorem}
  \label{thm:3t}
  Let $f$ and $g$ both satisfy~\eqref{eq:f}. Fix
  $u_o \in (\L1 \cap \BV)(\reali_+;\reali)$ and
  $u_b \in (\L1 \cap \BV)([0,T];\reali)$. Call $u$ and $v$ the
  solutions to the problems
  \begin{displaymath}
    % \label{eq:24}
    \!\!\!
    \left\{
      \begin{array}{lr@{\,}c@{\,}l}
        \partial_t u + \partial_x f (t,u) = 0
        & (t,x)
        & \in
        & [0,T] \times \reali_+
        \\
        u (0,x) = u_o (x)
        & x
        & \in
        & \reali_+
        \\
        u (t, 0) = u_b (t)
        & t
        & \in
        & [0,T]
      \end{array}
    \right.
    \mbox{ and } \
    \left\{
      \begin{array}{lr@{\,}c@{\,}l}
        \partial_t v + \partial_x g (t,v) = 0
        & (t,x)
        & \in
        & [0,T] \times \reali_+
        \\
        v (0,x) = u_o (x)
        & x
        & \in
        & \reali_+
        \\
        v (t, 0) = u_b (t)
        & t
        & \in
        & [0,T]
      \end{array}
    \right.
    \!\!
  \end{displaymath}
  constructed in Proposition~\ref{prop:2t}. Then, with
  $\mathcal{U} = \mathcal{U} (u_o, {u_b}_{|[0,t]})$ as
  in~\eqref{eq:19}, for all $t \in [0,T]$
  \begin{displaymath}
    % \label{eq:26}
    \begin{aligned}
      \norma{u (t) - v (t)}_{\L1 (\reali_+; \reali)} \leq \ & \max \!
      \left\{1, \norma{\partial_u g}_{\L\infty ([0,t] \times
          \mathcal{U}; \reali)}\right\} \, \norma{\partial_u (f -
        g)}_{\L\infty ([0,t] \times \mathcal{U};\reali)}
      \\
      & \times \left(\tv (u_o) + \tv\left(u_b; [0,t]\right) +
        \modulo{u_b (0+) - u_o (0+)} \right) \, t.
    \end{aligned}
  \end{displaymath}
\end{theorem}

\section{The Case of the Segment}
\label{sec:SG}

We consider here the case~\eqref{eq:1sg} where $x$ varies in a
segment. All statements are presented in details below, but proofs are
omitted since they are entirely analogous to the ones presented in
Section~\ref{sec:Tech}. The definition of solution to~\eqref{eq:1sg}
is given analogously to Definition~\ref{def:solsk}, adding an obvious
term related to the boundary $x=L$.

\begin{definition}
  \label{def:solskSG}
  A \emph{weak entropy solution} to the IBVP~\eqref{eq:1sg} on the
  interval $[0,T]$ is a map
  $u \in \L\infty\left([0,T] \times [0,L]; \reali \right)$, such that
  for any $k \in \reali$ and for any test function
  $\phi \in \Cc1 (\reali \times \reali; \reali_+)$ satisfies the
  following entropy inequalities
  \begin{displaymath}
    % \label{eq:ei1sg}
    \begin{aligned}
      & \int_0^T \int_0^L \left\{ \eta_k^+ \left(u (t,x)\right)
        \, \partial_t \phi (t,x) + \Phi_k^+\left(t,u (t,x)\right)
        \, \partial_x \phi (t,x) \right\} \d{x} \d{t}
      \\
      & + \int_0^L \eta_k^+ \left(u_o (x)\right) \, \phi (0,x) \d{x} -
      \int_0^L \eta_k^+ \left(u (T,x)\right) \, \phi (T,x) \d{x}
      \\
      & + \norma{\partial_u f}_{\L\infty ([0,T] \times
        \mathcal{U};\reali)} \left(\int_0^T \eta_k^+ \left(u_{b,1}
          (t)\right) \, \phi (t,0) \d{t} + \int_0^T \eta_k^+
        \left(u_{b,2} (t)\right) \, \phi (t,L) \d{t}\right) \geq 0,
    \end{aligned}
  \end{displaymath}
  and
  \begin{displaymath}
    % \label{eq:ei2sg}
    \begin{aligned}
      & \int_0^T \int_0^L \left\{ \eta_k^- \left(u (t,x)\right)
        \, \partial_t \phi (t,x) + \Phi_k^-\left(t,u (t,x)\right)
        \, \partial_x \phi (t,x) \right\} \d{x} \d{t}
      \\
      & + \int_0^L \eta_k^- \left(u_o (x)\right) \, \phi (0,x)
      \d{x} - \int_0^L \eta_k^- \left(u (T,x)\right) \, \phi
      (T,x) \d{x}
      \\
      & + \norma{\partial_u f}_{\L\infty ([0,T] \times
        \mathcal{U};\reali)} \left(\int_0^T \eta_k^- \left(u_{b,1}
          (t)\right) \, \phi (t,0) \d{t} + \int_0^T \eta_k^-
        \left(u_{b,2} (t)\right) \, \phi (t,L) \d{t}\right) \geq 0,
    \end{aligned}
  \end{displaymath}
where $\mathcal{U} = \mathcal{U} (u_o, {u_{b,1}}_{|[0,T]}, {u_{b,2}}_{|[0,T]})$ as
  in~\eqref{eq:19}.
\end{definition}

Throughout, we denote $\ub = \left(u_{b,1}, u_{b,2}\right)$ and
$\boldsymbol{w_b} = \left(w_{b,1}, w_{b,2}\right)$.

\begin{proposition}
  \label{prop:lipdepSG}
  Let $f \in \C1([0,T] \times \reali;\reali)$ be such that
  $\{u \mapsto \partial_t f(t.u)\} \in \Wloc{1}{\infty}
  (\reali;\reali)$
  for all $t \in [0,T]$,
  $u_o,w_o \in (\L1 \cap \L\infty)([0,L]; \reali)$ and
  $\ub, \boldsymbol{w_b} \in (\L1 \cap \L\infty)([0,T]; \reali^2)$.
  Let $u,w \in \L\infty ([0,L] \times [0,T]; \reali)$ solve the
  IBVP~\eqref{eq:1sg}, with data $(u_o,\ub)$ and
  $(w_o, \boldsymbol{w_b})$ respectively, in the sense of
  Definition~\ref{def:solskSG}, with $u$ and $w$ that both admit a trace
  for $x \to 0+$ and for $x \to L-$ for a.e.~$t \in [0,T]$. Then, for all
  $t \in [0,T]$,
  \begin{displaymath}
    % \label{eq:28}
    \norma{u (t) - w (t)}_{\L1 ([0,L]; \reali)}
    \leq
    \norma{u_o - w_o}_{\L1 (\reali_+[0,L]; \reali)}
    +
    \norma{\partial_u f}_{\L\infty ([0,t] \times \mathcal{U};\reali)}
    \sum_{i=1}^2
    \norma{u_{b,i} - w_{b,i}}_{\L1 ([0,t]; \reali)},
  \end{displaymath}
  where
  $\mathcal{U} = \mathcal{U} (\ub_{|[0,t]}, \, \boldsymbol{w_b}_{|[0,t]})$ is
  as in~\eqref{eq:19}.
\end{proposition}

Along the lines of the preceding sections, we present first the
results for a time independent flux and then those related to the non
autonomous case. We provide all those details where the present
results differ from those of sections~\ref{sec:autHL}
and~\ref{sec:nonautHL}.

\subsection{The Autonomous Case on the Segment}
\label{sec:autSG}

Consider the following autonomous IBVP, which is a particular case of~\eqref{eq:1sg}:
\begin{equation}
  \label{eq:1autSG}
  \left\{
    \begin{array}{l@{\qquad}r@{\,}c@{\,}l}
      \partial_t u + \partial_x f (u) = 0
      & (t,x)
      & \in
      & [0,T] \times [0,L]
      \\
      u (0,x) = u_o (x)
      & x
      & \in
      & [0,L]
      \\
      \left(u (t, 0), u (t,L)\right) = \ub (t)
      & t
      & \in
      & [0,T]\,.
    \end{array}
  \right.
\end{equation}
Solutions to~\eqref{eq:1autSG} are understood in the sense of
Definition~\ref{def:solskSG}. Observe that Proposition~\ref{prop:2SG}
applies to~\eqref{eq:1autSG}, under the hypothesis
$f \in \C1 (\reali;\reali)$.

The next Proposition ensures the existence of solutions to~\eqref{eq:1autSG}, as well as some of their properties, and it is the analogue to Proposition~\ref{prop:2}, with minor modifications in the estimates.

\begin{proposition}
  \label{prop:2SG}
  Let $f \in \Wloc{1}{\infty}(\reali; \reali)$,
  $u_o \in (\L1 \cap \BV) ([0,L]; \reali)$,
  $\ub \in (\L1 \cap \BV) ([0,T]; \reali^2)$. Then,
  problem~\eqref{eq:1autSG} admits a solution $u$ in the sense of
  Definition~\ref{def:solskSG}, with the properties:
  \begin{enumerate}
  \item If $u_o$ and $\ub$ are piecewise constant, then for $t$ small,
    the map $t \to u (t)$ coincides with the gluing of Lax solutions
    to Riemann problems at the points of jumps of $u_o$, at $x=0$ and
    at $x=L$.
  \item Range of $u$: with the notation in~\eqref{eq:19},
    $u (t,x) \in \mathcal{U} (u_o, {\ub}_{|[0,t]})$ for
    a.e.~$(t,x) \in [0,T] \times [0,L]$. Hence, for all $t \in [0,T]$,
    \begin{displaymath}
      \norma{u (t)}_{\L\infty ([0,L]; \reali)}
      \leq
      \max \left\{ \norma{u_o}_{\L\infty ([0,L]; \reali)}, \,
        \norma{u_{b,1}}_{\L\infty ([0,t];\reali)}, \,
        \norma{u_{b,2}}_{\L\infty ([0,t];\reali)}
      \right\}.
    \end{displaymath}
  \item $u$ is Lipschitz continuous in time: for all
    $t_1,t_2 \in [0,T]$,
    \begin{displaymath}
      \begin{array}{@{}c@{}}
        \norma{u (t_1) - u (t_2)}_{\L1 ([0,L]; \reali)}
        \leq
        C (u_o, \ub) \;
        \norma{f'}_{\L\infty (\mathcal{U};\reali)} \;
        \modulo{t_2 - t_1},
        \\
        \!\!\!
        C (u_o, \ub)
        =
        \tv (u_o)
        + \tv (\ub; [0,t_1\vee t_2])
        + \modulo{u_{b,1} (0+) - u_o (0+)}
        + \modulo{u_{b,2} (0+) - u_o (L-)}
      \end{array}
    \end{displaymath}
    where
    $\mathcal{U} = \mathcal{U} (u_o, {\ub}_{|[0,t_1 \vee t_2]})$, with
    the notation~\eqref{eq:19}.
  \item Total variation estimate: for all $t \in [0,T]$
    \begin{displaymath}
      \tv\left(u (t) \right)
      \leq
      \tv (u_o)
      +
      \tv \left(\ub; [0,t]\right)
      +
      \modulo{u_{b,1}(0+) - u_o (0+)}
      +\modulo{u_{b,2} (0+) - u_o (L-)}\,.
    \end{displaymath}
  \end{enumerate}
\end{proposition}

\noindent We conclude this Section stating the stability of the
solutions to~\eqref{eq:1autSG} with respect to the flux, similarly to
Theorem~\ref{thm:3}.

\begin{theorem}
  \label{thm:3SG}
  Let $f,g \in \C1 (\reali; \reali)$,
  $u_o \in (\L1 \cap \BV) ([0,L]; \reali)$ and
  $\ub \in (\L1 \cap \BV) ([0,T]; \reali^2)$. Call $u$ and $v$ the
  solutions to the IBVP~\eqref{eq:1autSG}, with flux $f$ and $g$
  respectively, constructed in Proposition~\ref{prop:2SG}. Then, with
  $\mathcal{U} = \mathcal{U} (u_o, {\ub}_{|[0,t]})$ as
  in~\eqref{eq:19}, for all $t \in [0,T]$,
  \begin{align*}
    &
      \norma{u (t) - v (t)}_{\L1 ([0,L]; \reali)}
      \\
    \leq \
    & \max \!
      \left\{1, \norma{g'}_{\L\infty (\mathcal{U}; \reali)}\right\} \,
      \norma{D(f - g)}_{\L\infty (\mathcal{U};\reali)}
    \\
    & \times \left(\tv (u_o) + \tv\left(\ub; [0,t]\right) +
        \modulo{u_{b,1} (0+) - u_o (0+)} +
      \modulo{u_{b,2} (0+) - u_o (L-)}\right) \, t.
      \end{align*}
\end{theorem}

\subsection{The Non Autonomous Case on the Segment}
\label{sec:nonautSG}

We now extend the results obtained in Section~\ref{sec:autSG} to problem~\eqref{eq:1sg}.

\begin{proposition}
  \label{prop:2tSG}
  Let $f$ satisfy~\eqref{eq:f}.
  Fix $u_o \in (\L1 \cap \BV) ([0,L]; \reali)$ and
  $\ub \in (\L1 \cap \BV) ([0,T]; \reali^2)$. Then,
  problem~\eqref{eq:1sg} admits a solution $u$ in the sense of
  Definition~\ref{def:solskSG}, with:
  \begin{enumerate}
  \item Range of $u$: with the notation in~\eqref{eq:19},
    $u (t,x) \in \mathcal{U} (u_o, {\ub}_{|[0,t]})$ for
    a.e.~$(t,x) \in [0,T] \times [0,L]$. Hence, for all
    $t \in [0,T]$,
    \begin{displaymath}
      \norma{u (t)}_{\L\infty ([0,L]; \reali)}
      \leq
      \max \left\{ \norma{u_o}_{\L\infty ([0,L]; \reali)}, \,
        \norma{u_{b,1}}_{\L\infty ([0,t];\reali)}, \,
        \norma{u_{b,2}}_{\L\infty ([0,t];\reali)}
      \right\}.
    \end{displaymath}
  \item $u$ is Lipschitz continuous in time: for all
    $t_1,t_2 \in [0,T]$,
    \begin{displaymath}
      \begin{array}{@{}c@{}}
        \norma{u (t_1) - u (t_2)}_{\L1 ([0,L]; \reali)}
        \leq
        C (u_o, \ub) \;
        \norma{\partial_u f}_{\L\infty ([0,t_1 \vee t_2] \times \mathcal{U};\reali)} \;
        \modulo{t_2 - t_1},
        \\
        \!\!\!
        C (u_o, \ub)
        =
        \tv (u_o) + \tv (\ub; [0,t_1\vee t_2])
        + \modulo{u_{b,1} (0+) - u_o (0+)}
        + \modulo{u_{b,2} (0+)
        - u_o (L-)}
      \end{array}
    \end{displaymath}
    where $\mathcal{U} = \mathcal{U} (u_o, {\ub}_{|[0,t_1 \vee t_2]})$,
    with the notation~\eqref{eq:19}.
  \item Total variation estimate: for all $t \in [0,T]$
    \begin{displaymath}
      \tv\left(u (t) \right)
      \leq
      \tv (u_o)
      +
      \tv \left(\ub; [0,t]\right)
      +
      \modulo{u_{b,1}(0+) - u_o (0+)}
      +
      \modulo{u_{b,2} (0+) - u_o (L-)}\,.
    \end{displaymath}
  \end{enumerate}
\end{proposition}

\noindent We conclude this section with the analogue to
Theorem~\ref{thm:3t}, i.e.~the stability of the solution
to~\eqref{eq:1sg} with respect to the flux.

\begin{theorem}
  \label{thm:3tSG}
  Let $f,g$ satisfy~\eqref{eq:f}. Fix
  $u_o \in (\L1 \cap \BV)([0,L];\reali)$ and
  $\ub \in (\L1 \cap \BV)([0,T];\reali^2)$. Call $u, v$ the
  solutions to the IBVP~\eqref{eq:1sg}, with flux $f$ and $g$
  respectively, constructed in Proposition~\ref{prop:2tSG}. Then, with
  $\mathcal{U} = \mathcal{U} (u_o, {\ub}_{|[0,t]})$ as
  in~\eqref{eq:19}, for all $t \in [0,T]$,
  \begin{align*}
    & \norma{u (t) - v (t)}_{\L1 ([0,L]; \reali)}
    \\
    \leq \ &
             \max \!
             \left\{1,
             \norma{\partial_u g}_{\L\infty ([0,t] \times \mathcal{U}; \reali)}\right\} \,
             \norma{\partial_u (f - g)}_{\L\infty ([0,t] \times \mathcal{U};\reali)}
    \\
    & \times \left(\tv (u_o) + \tv\left(\ub; [0,t]\right) +
      \modulo{u_{b,1} (0+) - u_o (0+)} +
      \modulo{u_{b,2} (0+) - u_o (L-)}\right) \, t.
  \end{align*}
\end{theorem}

\section{Technical Proofs}
\label{sec:Tech}

We distinguish between \emph{classical entropy-entropy flux pair} and
\emph{boundary entropy-entropy flux pair}. In similar settings, the
former notion, in the time independent case, is given
in~\cite[Paragraph~7.4]{DafermosBook} or~\cite[Chapter~2,
Definition~3.22]{MalekEtAlBook}, while for the latter we refer
to~\cite{OttoPhD, OttoCR}, see also~\cite[Chapter~2,
Definition~7.1]{MalekEtAlBook}, \cite[Definition~2]{Martin}
and~\cite[Definition~2]{Vovelle}. We provide below the explicit
definitions in the case of interest here, where $f = f(t,u)$.

\begin{definition}
  \label{def:ceef}
  The pair
  $(\eta, q) \in \C1 (\reali; \reali) \times \C1 ([0,T] \times
  \reali; \reali)$
  is called a \emph{classical entropy-entropy flux pair} for the flux
  $f \in \C1 ([0,T] \times \reali; \reali)$ if:
  \begin{enumerate}
  \item $\eta$ is convex;
  \item for all $t \in [0,T]$ and all $u \in \reali$,
    $\partial_u q (t,u) = \eta' (u) \; \partial_u f(t,u)$.
  \end{enumerate}
\end{definition}

\begin{definition}
  \label{def:beef}
  The pair
  $(H,Q) \in \C1(\reali^2;\reali) \times \C1([0,T] \times \reali^2;
  \reali)$
  is called a \emph{boundary entropy-entropy flux pair} for the flux
  $f \in \C1 ([0,T] \times \reali; \reali)$ if:
  \begin{enumerate}
  \item for all $w \in \reali$, the function $u \mapsto H(u,w)$ is
    convex;
  \item for all $t \in [0,T]$ and all $u,w \in \reali$,
    $\partial_u Q(t,u,w) = \partial_u H(u,w) \; \partial_u
    f(t,u)$;
  \item for all $t \in [0,T]$ and all $w \in \reali$, $H(w,w)=0$,
    $Q(t,w,w)=0$ and $\partial_u H(w,w)=0$.
  \end{enumerate}
\end{definition}

\noindent Consequences of Definition~\ref{def:solsk} are
collected in the following lemmas, whose proofs directly follow
from~\cite[Lemma~1 and Remark~3]{Vovelle}, see also~\cite[Lemma~3, Lemma~4
and Lemma~16]{Martin}.

\begin{lemma}
  \label{lem:solint}
  If $u \in \L\infty([0,T] \times \reali_+;\reali)$ is a weak entropy
  solution to~\eqref{eq:1} in the sense of Definition~\ref{def:solsk},
  then, for all classical entropy-entropy flux pairs $(\eta,q)$, for
  all $\phi \in \Cc1(\reali\times \pint\reali_+;\reali_+)$,
  \begin{equation}
    \label{eq:solint}
    \begin{aligned}
      \int_0^T\int_{\reali_+} \left\{\eta\left(u(t,x)\right)
        \, \partial_t \phi(t,x) + q\left(t,u(t,x)\right) \, \partial_x
        \phi (t,x)\right\} \d{x} \d{t} &
      \\
      + \int_{\reali_+} \eta\left(u_o(x)\right) \, \phi(0,x) \d{x} -
      \int_{\reali_+} \eta\left(u(T,x)\right) \, \phi(T,x) \d{x} &
      \geq 0.
    \end{aligned}
  \end{equation}
  In particular, for all
  $\phi \in \Cc1(\reali\times \pint\reali_+;\reali_+)$ and for all
  $k \in \reali$,
  \begin{equation}
    \label{eq:solKr}
    \begin{aligned}
      \int_0^T \!\! \int_{\reali_+}\!\!  & \left\{\modulo{u(t,x)-k}
        \, \partial_t \phi(t,x) \right.
      \\
      & \left.  +
        \sgn\left(u(t,x)-k\right)\left(f\left(t,u(t,x)\right) - f(t,k)
        \right) \partial_x \phi (t,x)\right\} \d{x} \d{t}
      \\
      + \int_{\reali_+} & \modulo{u_o(x)-k} \phi(0,x) \d{x} -
      \int_{\reali_+} \modulo{u(T,x)-k} \phi(T,x) \d{x} \geq 0.
    \end{aligned}
  \end{equation}
\end{lemma}

\begin{lemma}
  \label{lem:bc}
  If $u \in \L\infty([0,T] \times \reali_+;\reali)$ is a weak entropy
  solution to~\eqref{eq:1} in the sense of Definition~\ref{def:solsk},
  then, for all boundary entropy-entropy flux pair $(H,Q)$ and for all
  $\beta \in \L1(\reali;\reali)$ with $\beta \geq 0$ a.e.,
  \begin{equation}
    \label{eq:bc}
    \esslim_{s \to 0^+}
    \int_0^T Q\left(t,u(t,s),u_b(t)\right) \, \beta(t) \d{t}
    \leq 0.
  \end{equation}
  Moreover, if $u$ admits a trace $u (t, 0+)$ at $x=0$ for
  a.e.~$t \in [0,T]$, \eqref{eq:bc} is equivalent to
  \begin{equation}
    \label{eq:bctr}
    \int_0^T Q\left(t,u(t,0+),u_b(t)\right) \, \beta(t) \d{t} \leq 0.
  \end{equation}
\end{lemma}

We now extend part of~\cite[Chapter~2, Lemma~7.24]{MalekEtAlBook} to the
time dependent case.

\begin{lemma}
  \label{lem:equiv}
  Let $u_b \in \L\infty ([0,T]; \reali)$ and let
  $u \in \L\infty ([0,T] \times \reali; \reali)$ admit a trace
  $u (t, 0+)$ at $x=0$ for a.e.~$t \in [0,T]$. If~\eqref{eq:bctr}
  holds, then for a.e.~$t \in [0,T]$ and for all
  $k \in \mathcal{I} \left(u (t, 0+), u_b (t)\right)$ as
  in~\eqref{eq:5},
  \begin{equation}
    \label{eq:4}
    \sgn \left(u (t, 0+) - u_b (t)\right)
    \left(
      f\left(t, u (t, 0+)\right) - f (t,k)
    \right)
    \leq 0 \,.
  \end{equation}
\end{lemma}

\begin{proof}
  For all $k \in \reali$ and for $n \in \naturali \setminus\{0\}$,
  define the maps
  \begin{equation}
    \label{eq:11}
    \begin{aligned}
      \Delta^k (u,w) = \ & \min_{z \in \mathcal{I} (w,k)} \modulo{u -
        z}
      \\
      \mathcal{F}^k (t, u, w) = \ &
      \begin{cases}%{l@{\qquad}c}
        f (t,w) - f (t,u) & \mbox{for } u \leq w \leq k
        \\
        0 & \mbox{for } w \leq u \leq k
        \\
        f (t,u) - f (t,k) & \mbox{for } w \leq k \leq u
        \\
        f (t,k) - f (t,u) & \mbox{for } u \leq k \leq w
        \\
        0 & \mbox{for } k \leq u \leq w
        \\
        f (t,u) - f (t,w) & \mbox{for } k \leq w \leq u
      \end{cases}
      % \right.
      \\
      H_n^k (u,w) = \ & \left( \left( \Delta^k (u,w) \right)^2 +
        \dfrac{1}{n^2} \right)^{1/2} - \dfrac{1}{n} \\%[10pt]
      Q_n^k (t,u,w) = \ & \int_w^u \partial_u H_n^k (z,w)
      \; \partial_u f (t,z) \d{z} .
    \end{aligned}
  \end{equation}
  Clearly, for all $k \in \reali$, the sequence of boundary
  entropy-entropy flux pairs $(H_n^k, Q_n^k)$ converges uniformly to
  $(\Delta^k, \mathcal{F}^k)$ as $n \to
  +\infty$.
  Applying~\eqref{eq:bctr} with $Q$ replaced by $Q_n^k$, in the limit
  $n \to +\infty$ yields that for all $k \in \reali$ and for all
  $\beta \in \L1 (\reali; \reali)$ with $\beta \geq 0$ a.e.,
  \begin{eqnarray}
    \nonumber
    \int_0^T
    \mathcal{F}^k \left(t,u(t,0+),u_b(t)\right) \, \beta(t) \d{t}
    & \leq
    & 0
    \\
    \label{eq:15}
    \mathcal{F}^k \left(t,u(t,0+),u_b(t)\right)
    & \leq
    & 0
      \qquad \mbox{ for a.e. } t \in [0,T] \,.
  \end{eqnarray}
  Choose now $k \in \mathcal{I}\left(u (t,0+) , u_b (t)\right)$ so
  that, by~\eqref{eq:11}, the bound~\eqref{eq:15} ensures~\eqref{eq:4}.
\end{proof}

\begin{proofof}{Proposition~\ref{prop:lipdep}}
This proof closely follows that of~\cite[Theorem~4.3]{bordo}, but
using the doubling of variables method as in~\cite[Lemma~17]{Martin},
which is consistent with the present Definition~\ref{def:solsk}.  Key
points are the choice of an appropriate test function and the use of
Lemma~\ref{lem:bc} and Lemma~\ref{lem:equiv}.

Note that here there is no source term, the flux $f$ does not depend
on the space variable and we are dealing with $\reali_+$ instead of a
bounded domain $\Omega \subseteq \reali^n$. A careful checking of the
proof in~\cite{bordo} shows that the present assumptions on $f$ are
sufficient.
\end{proofof}

\subsection{Proofs related to the Autonomous IBVP on the Half--Line}
\label{subs:AHL}

\begin{proofof}{Proposition~\ref{prop:2}}
  For $\epsilon>0$, introduce the set
  $\PC (\reali_+; \epsilon\interi)$ of maps $u$ of the form
  $u = \sum_{\alpha = 1}^N u_\alpha \, \caratt{I_\alpha}$, where
  $N \in \naturali$, $u_\alpha \in \epsilon \interi$ and $I_\alpha$ is
  a real interval for all $\alpha = 1, \ldots, N$.
  $\PLC (\reali; \reali)$ is the set of real valued piecewise linear
  and continuous functions defined on $\reali$.

  \paragraph{A.1)~Construction of $\epsilon$--approximate solutions.}
  Following~\cite[Chapter~6]{BressanLectureNotes}, for any positive
  $\epsilon$ introduce the following approximations:
  \begin{equation}
    \label{eq:app}
    \begin{array}{rcl@{\mbox{ such that }}l}
      u_o^\epsilon & \in & \PC (\reali_+; \epsilon\interi)
      &
        \left\{
        \begin{array}{l}
          \norma{u_o^\epsilon}_{\L\infty (\reali_+;\reali)}
          \leq
          \norma{u_o}_{\L\infty (\reali_+;\reali)}
          \\
          \tv (u_o^\epsilon) \leq\tv (u_o)
          \\
          \lim_{\epsilon \to 0} \norma{u_o^\epsilon-u_o}_{\L1 (\reali_+; \reali)} =0
        \end{array}
      \right.
      \\
      u_b^\epsilon & \in & \PC ([0,T]; \epsilon\interi)
      & \left\{
        \begin{array}{l}
          \norma{u_b^\epsilon}_{\L\infty ([0,t];\reali)}
          \leq
          \norma{u_b}_{\L\infty ([0,t];\reali)}
          \quad \forall \, t \in [0,T]
          \\
          \tv (u_b^\epsilon; [0,t]) \leq \tv (u_b; [0,t])
          \quad \forall \, t \in [0,T]
          \\
          \lim_{\epsilon \to 0} \norma{u_b^\epsilon-u_b}_{\L1 ([0,T]; \reali)} =0
          \\
          \modulo{u^\epsilon_b(0+) - u^\epsilon_o(0+)} \leq \modulo{u_b(0+) - u_o(0+)}
        \end{array}
      \right.
      \\
      f^\epsilon & \in & \PLC (\reali; \reali)
      &
        \left\{
        \begin{array}{l}
          f^\epsilon (u) = f (u)
          \mbox{ for all } u \in \epsilon\interi
          \\
          f^\epsilon_{\left|]k\epsilon,(k+1)\epsilon[\right.}
          \mbox{ is an affine function for all } k \in \interi
          \\
          \norma{D f^\epsilon}_{\L\infty(\mathcal{U};\reali)}
          \leq
          \norma{f'}_{\L\infty(\mathcal{U};\reali)} \,.
        \end{array}
      \right.
    \end{array}
  \end{equation}

  \noindent We approximate the solution to the original
  IBVP~\eqref{eq:1aut} with exact solutions $u_\epsilon$ to the
  $\epsilon$--approximate~IBVPs
  \begin{equation}
    \label{eq:6}
    \left\{
      \begin{array}{l@{\qquad\qquad}r@{\,}c@{\,}l}
        \partial_t u^\epsilon + \partial_x f^\epsilon (u^\epsilon) = 0
        & (t,x)
        & \in
        & [0,T] \times \reali_+
        \\
        u^\epsilon (0,x) = u_o^\epsilon (x)
        & x
        & \in
        & \reali_+
        \\
        u^\epsilon (t, 0) = u_b^\epsilon (t)
        & t
        & \in
        & [0,T] \,.
      \end{array}
    \right.
  \end{equation}
  At the initial time $t=0$, solving~\eqref{eq:6} for $x>0$ amounts to
  glue the solutions to the Riemann problems at the points of jump in
  $u_o^\epsilon$, see~\cite[\S~6.1]{BressanLectureNotes}. A local
  solution at $(0,0)$ is obtained by restricting the solution to the
  Riemann problem for $f^\epsilon$ with left and right state
  $u_b^\epsilon(0)$ and $u_o^\epsilon(0)$ respectively,
  see~\cite[Example~C]{AmadoriColombo97}. Recall
  from~\cite[Chapter~6]{BressanLectureNotes} that, with the above
  choice of $f^\epsilon$, the solutions to Riemann problems with data
  in $\epsilon\interi$ still take values in the set $\epsilon\interi$.

  We thus have a piecewise constant solution $u^\epsilon$
  to~\eqref{eq:6} defined for $t>0$ sufficiently small. This solution
  can be prolonged up to the first \emph{time of interaction} $t_1>0$
  at which one of the following events takes place:
  \begin{enumerate}[(i)]
    \setlength{\itemsep}{0pt} \setlength{\parskip}{0pt}
  \item two or more lines of discontinuity hit each other;
  \item one wave hits the boundary $x=0$;
  \item the value of the boundary condition $u_b^\epsilon$ changes.
  \end{enumerate}
  In case~(i), it is possible to extend the solution beyond $t_1$ by
  solving the new Riemann problems generated by the interactions, as
  in~\cite[\S~6.1]{BressanLectureNotes}. In cases~(ii) and~(iii), the
  extension beyond $t_1$ is achieved by restricting to $\reali_+$ the
  solution to the Riemann problem with left state
  $u_b^\epsilon (t_1+)$ and right state $u^\epsilon (t_1, 0+)$. The
  solution is then prolonged up to the next time of interaction
  $t_2> t_1$, and so on.

  Note that, by construction, waves in $u^\epsilon$ satisfy both
  Rankine--Hugoniot condition~\cite[Formula~(4.3.5)]{DafermosBook} and
  Oleinik entropy condition~\cite[Formula~(8.4.3)]{DafermosBook}, in
  the sense that, whenever two states $u^\ell$ and $u^r$ in
  $u^\epsilon$ are separated by a wave propagating with speed
  $\lambda$, we have
  \begin{equation}
    \label{eq:17}
    \lambda = \frac{f^\epsilon (u^\ell) - f^\epsilon (u^r)}{u^\ell
      - u^r}
    \mbox{ and }
    \frac{f^\epsilon (u^r) - f^\epsilon (k)}{u^r - k} \leq \lambda
    \leq \frac{f^\epsilon (k) - f^\epsilon (u^\ell)}{k - u^\ell}
    \quad \forall k \in\mathcal{I} (u^\ell, u^r) \,.
  \end{equation}
  Moreover, the above conditions~\eqref{eq:17} impose that, whenever
  $u^\epsilon (t, 0+) = u^r$, we have that
  \begin{equation}
    \label{eq:18}
    \mbox{if } u^r \neq u_b^\epsilon (t)
    \quad \mbox{ then } \quad
    \dfrac{f^\epsilon (u^r) - f^\epsilon (k)}{u^r - k}
    \leq
    0
    \quad
    \forall k \in \mathcal{I} (u_b^\epsilon (t), u^r) \,.
  \end{equation}

  \paragraph{A.2)~Wave Front Tracking Solutions are Weak Entropy
    Solutions.}

  By standard arguments, it is sufficient to verify~\eqref{eq:ei1}
  and~\eqref{eq:ei2} in the following two cases:

  \smallskip

  \noindent%
  \textbf{1.} The support of the (positive) test function $\phi$ is
  contained in
  $[t_1, t_2] \times [x_1,x_2] \subset \left]0,T\right[ \times
  \pint{\reali}_+$
  and here the wave front tracking solution $u^\epsilon$ attains only
  the two values $u^\ell$ and $u^r$, separated by a wave with speed
  $\lambda = \frac{x_2-x_1}{t_2-t_1}$.

  \smallskip

  \noindent%
  \textbf{2.} The support of the (positive) test function $\phi$ is
  contained in $[t_1, t_2] \times [x_1, x_2]$ with $x_1 < 0 < x_2$,
  the boundary data satisfies $u_b^\epsilon (t) = u^\ell$ for
  $t \in [t_1, t_2]$ and $u^\epsilon (t,x) = u^r$ for
  $(t,x) \in [t_1, t_2] \times \left]0, x_2\right]$.

  \smallskip

  The other cases, that of a single wave with negative speed,
  of interacting waves, of waves interacting with the boundary and of
  the boundary datum changing value, can be recovered through
  manipulations of the test functions and immediate modifications
  of~\textbf{1.} and~\textbf{2}.

  \smallskip

  \noindent%
  \textbf{1.} Assume $k < u^\ell < u^r$. Then, direct computations
  show that~\eqref{eq:ei1} is equivalent to
  \begin{equation}
    \label{eq:caso1}
    \left[\lambda (u^r - u^\ell) - \left(f^\epsilon (u^r) - f^\epsilon (u^\ell)\right)\right]
    \int_{t_1}^{t_2} \phi \left(t, x_1 + \lambda (t-t_1)\right) \d{t}
    \geq 0 \,,
  \end{equation}
  which holds since the left hand side vanishes by the
  Rankine--Hugoniot condition~\eqref{eq:17}. It is immediate to check
  that the left hand side in~\eqref{eq:ei2} vanishes.

  If $u^\ell < k < u^r$, then~\eqref{eq:ei1} is equivalent to
  \begin{displaymath}
    (u^r - k) \left(\lambda - \dfrac{f^\epsilon (u^\ell) - f^\epsilon (u^r)}{u^\ell - u^r}\right)
    \int_{t_1}^{t_2} \phi \left(t, x_1 + \lambda (t-t_1)\right) \d{t}
    \geq 0 \,,
  \end{displaymath}
  which holds by Oleinik entropy condition~\eqref{eq:17}. On the other
  hand, \eqref{eq:ei2} is equivalent to
  \begin{displaymath}
    (k - u^\ell)
    \left(\dfrac{f^\epsilon (k) - f^\epsilon (u^\ell)}{k - u^\ell} - \lambda\right)
    \int_{t_1}^{t_2} \phi \left(t, x_1 + \lambda (t-t_1)\right) \d{t}
    \geq 0 \,,
  \end{displaymath}
  which again holds by Oleinik entropy condition~\eqref{eq:17}.

  If $u^\ell < u^r < k$, then~\eqref{eq:ei2} is equivalent
  to~\eqref{eq:caso1}, while the left hand side in~\eqref{eq:ei1}
  vanishes.

  The cases $k < u^r < u^\ell$, $u^r < k < u^\ell$ and
  $u^r < u^\ell < k$ are entirely analogous.

  \smallskip

  \noindent%
  \textbf{2.} Assume $k < u^\ell < u^r$. Then, direct computations
  show that~\eqref{eq:ei1} is equivalent to
  \begin{displaymath}
    \left[
      -\left(f^\epsilon (u^r) - f^\epsilon (k)\right)
      +
      \norma{Df^\epsilon}_{\L\infty (\mathcal{U}; \reali)} \, (u^\ell - k)
    \right]
    \int_{t_1}^{t_2} \phi (t, 0) \d{t}
    \geq 0 \,.
  \end{displaymath}
  Note that, by the Lipschitz continuity of $f^\epsilon$, we have
  \begin{displaymath}
    -\left(f^\epsilon (u^r) - f^\epsilon (k)\right)
    +
    \norma{Df^\epsilon}_{\L\infty (\mathcal{U}; \reali)} \, (u^\ell - k)
    \geq
    f^\epsilon (u^\ell) - f^\epsilon (u^r)
  \end{displaymath}
  and the latter term above is non negative by~\eqref{eq:18}. Hence,
  the left hand side in~\eqref{eq:ei2} vanishes.

  Assume $u^\ell < k < u^r$. Then, \eqref{eq:ei1} is equivalent to
  \begin{displaymath}
    \left(f^\epsilon (k) - f^\epsilon (u^r)\right)
    \int_{t_1}^{t_2} \phi (t, 0) \d{t}
    \geq 0 \,,
  \end{displaymath}
  which holds by Oleinik entropy condition~\eqref{eq:18}. Hence,
  \eqref{eq:ei2} reads
  \begin{displaymath}
    \norma{Df^\epsilon}_{\L\infty (\mathcal{U};\reali)} \,
    (k - u^\ell)
    \int_{t_1}^{t_2} \phi (t, 0) \d{t}
    \geq
    0
  \end{displaymath}
  and this inequality clearly holds.

  Assume $u^\ell < u^r < k$. Then, the left hand side in~\eqref{eq:ei1}
  vanishes. Condition~\eqref{eq:ei2} becomes
  \begin{displaymath}
    \left[
      \left(f^\epsilon (u^r) - f^\epsilon (k)\right)
      +
      \norma{Df^\epsilon}_{\L\infty (\mathcal{U}; \reali)} \, (k - u^\ell)
    \right]
    \int_{t_1}^{t_2} \phi (t, 0) \d{t}
    \geq 0 \,.
  \end{displaymath}
  Note that, by the Lipschitz continuity of $f^\epsilon$, we have
  \begin{displaymath}
    \left(f^\epsilon (u^r) - f^\epsilon (k)\right)
    +
    \norma{Df^\epsilon}_{\L\infty (\mathcal{U}; \reali)} \, (k - u^\ell)
    \geq
    \norma{Df^\epsilon}_{\L\infty (\mathcal{U}; \reali)} (u^r - u^\ell)
  \end{displaymath}
  and the latter term above is non negative in the present case.

  \paragraph{A.3)~The map $t \to \tv\left(u^\epsilon (t)\right)$ is
    uniformly bounded, as long as $u^\epsilon$ is defined.} Introduce
  for $t \in [0,T]$, the Glimm functional
  \begin{equation}
    \label{eq:Vepsilon}
    V^\epsilon (t)
    =
    \tv\left(u^\epsilon (t)\right)
    +
    \tv\left(u^\epsilon_b ; [t,T]\right)
    +
    \modulo{u_b^\epsilon(t+) - u^\epsilon (t,0+)} \,.
  \end{equation}
  Clearly, $\tv\left(u^\epsilon (t)\right) \leq V^\epsilon (t)$.  We
  claim that $t \to V^\epsilon (t)$ is non increasing. Indeed, at an
  interaction time, the proof in~\cite[\S~6.1]{BressanLectureNotes}
  applies in case~(i) in Step~\textbf{A.1}, while minor modifications
  yield the proof in the other two cases~(ii) and~(iii). The
  inequality $V^\epsilon (t) \leq V^\epsilon (0)$ implies
  \begin{equation}
    \label{eq:12}
    \tv\left(u^\epsilon (t)\right)
    \leq
    \tv (u^\epsilon_o)
    +
    \tv \left(u^\epsilon_b; [0,t]\right)
    +
    \modulo{u^\epsilon_b (0+) - u^\epsilon_o (0+)} \,.
  \end{equation}

  \paragraph{A.4)~The total number of interactions is finite and
    $u^\epsilon$ is defined for all $t \in [0,T]$.} When $t$ is not
  a time of interaction, define the weighted number of discontinuities
  in $u^\epsilon (t)$ as
  \begin{eqnarray*}
    \sharp (t)
    & = &
          \left[\mbox{number of discontinuities in }u^\epsilon (t)\right]
    \\
    & &
        +
        2\frac{\norma{u_b}_{\L\infty (\reali_+; \reali)}
        +
        \norma{u_o}_{\L\infty (\reali_+; \reali)}}{\epsilon} \,
        \left[\mbox{number of discontinuities in }
        {\mathscr{T}}_t u^\epsilon_b \right]
    \\
    & &
        +
        \frac{1}{\epsilon} \, \modulo{u_b^\epsilon (t) - u^\epsilon (t, 0)} \,,
  \end{eqnarray*}
  where we used the notation~\eqref{eq:Trasl}. If $t$ is an
  interaction time, set $\sharp(t) = \lim_{\tau \to t+} \sharp(\tau)$.

  The procedure in~\cite[\S~6.1]{BressanLectureNotes} can be applied,
  ensuring that at those interaction times where $\sharp$ increases,
  $V^\epsilon$ diminishes by at least $\epsilon$.

  \paragraph{A.5)~Range of $u^\epsilon$.} At any interaction time $t_*$,
  the new values attained by $u^\epsilon$ lie in the convex hull of
  the values attained by $u^\epsilon$ before time $t_*$, proving that
  $u^\epsilon (t, x) \subseteq \mathcal{U} (u_o, {u_b}_{|[0,t]})$ for
  a.e.~$(t,x) \in [0,T] \times \reali_+$, with the
  notation~\eqref{eq:19}. It is then immediate to verify that at any
  time $t \in [0,T]$
  \begin{equation}
    \label{eq:13}
    \norma{u^\epsilon (t)}_{\L\infty (\reali_+; \reali)}
    \leq
    \max \left\{
      \norma{u_o}_{\L\infty (\reali_+;\reali)}, \,
      \norma{u_b}_{\L\infty ([0,t];\reali)}
    \right\} \,.
  \end{equation}

  \paragraph{A.6)~$\L1$--Lipschitz continuity of $t \to u^\epsilon (t)$.}
  Assume that $t_2 > t_1$.  Observe first that $u^\epsilon$ remains
  unaltered on the interval $[0, t_2]$ if $u_b^\epsilon$ is
  substituted by
  $\tilde u_b^\epsilon = u_b^\epsilon \, \caratt{[0, t_2]} +
  u_b^\epsilon (t_2+) \, \caratt{[t_2, T]}$.
  At any interaction time $t_*$, if $t_1 < t_* < t_2$ and $t_2 - t_1$
  is sufficiently small, denoting
  $\mathcal{U} = \mathcal{U} (u_o, {u_b}_{|[0, t_2]})$ as
  in~\eqref{eq:19},
  \begin{align}
    \nonumber
    &
      \norma{u^\epsilon (t_2) - u^\epsilon (t_1)}_{\L1 (\reali_+; \reali)}
    \\
    \nonumber
    \leq  \ &
              \norma{f'}_{\L\infty (\mathcal{U}; \reali)} \;
              V^\epsilon (t_*) \; \modulo{t_2 - t_1}
              \qquad\qquad\qquad\qquad\,[\mbox{by \cite[Formula~(6.14)]{BressanLectureNotes} and~\eqref{eq:12}}]
    \\
    \nonumber
    \leq \
& \norma{f'}_{\L\infty (\mathcal{U}; \reali)} \;
             V^\epsilon (0) \; \modulo{t_2 - t_1}
             \qquad\qquad\qquad\quad\quad\;\;[\mbox{since $t \to V^\epsilon (t)$ is non increasing}]
    \\
    \nonumber
    = \
    & \norma{f'}_{\L\infty (\mathcal{U}; \reali)} \;
          \left(
          \tv(u_o^\epsilon)
          +
          \tv(\tilde u^\epsilon_b)
          +
          \modulo{\tilde u_b^\epsilon(0+) - u^\epsilon_o (0+)}
          \right)
          \modulo{t_2 - t_1}
          \qquad\quad\hfill[\mbox{by~\eqref{eq:Vepsilon}}]
    \\
    \label{eq:7}
    = \
    & \norma{f'}_{\L\infty (\mathcal{U}; \reali)} \;
          \left(
          \tv(u_o^\epsilon)
          +
          \tv(u^\epsilon_b;[0,t_2])
          +
          \modulo{u_b^\epsilon(0+) - u^\epsilon_o (0+)}
          \right)
          \modulo{t_2 - t_1} \,.
  \end{align}

  \paragraph{A.7)~Existence of a Solution.}
  By Helly Theorem~\cite[Theorem~2.4]{BressanLectureNotes}, for any
  sequence $\epsilon_n$ converging to $0$, the sequence
  $u^{\epsilon_n}$ converges pointwise almost everywhere, up to a
  subsequence, to a map $u$. We now show that this limit function $u$
  is a weak entropy solution to~\eqref{eq:1aut}, in the sense of
  Definition~\ref{def:solsk}. Any $u^{\epsilon_n}$ is a weak entropy
  solution to~\eqref{eq:6} by Step~\textbf{A.2}; hence,
  $u^{\epsilon_n}$ satisfies for any $k \in \reali$ and for any test
  function $\phi \in \Cc1 (\reali \times \reali;\reali)$ the two
  entropy inequalities
  \begin{align}
    \label{eq:r1}
    0 \leq \
    & \int_0^T \int_{\reali_+}
      \left\{ \eta_k^\pm \left(u^{\epsilon_n} (t,x)\right)
      \, \partial_t \phi (t,x) +
      \Phi_{k,n}^\pm\left(u^{\epsilon_n} (t,x)\right)
      \, \partial_x \phi (t,x) \right\} \d{x} \d{t}
    \\
    \label{eq:r2}
    & + \int_{\reali_+} \eta_k^\pm \left(u_o^{\epsilon_n} (x)\right) \, \phi (0,x)
      \d{x} - \int_{\reali_+} \eta_k^\pm \left(u^{\epsilon_n} (T,x)\right) \, \phi
      (T,x) \d{x}
    \\
    \label{eq:r3}
    & + \norma{Df^{\epsilon_n}}_{\L\infty (\mathcal{U};\reali)}
      \int_0^T \eta_k^\pm \left(u_b^{\epsilon_n} (t)\right) \,
      \phi (t,0) \d{t},
  \end{align}
  where $\eta_k^{\pm}$ and $\Phi_{k,n}^{\pm}$ are defined as
  in~\eqref{eq:sken}, using the autonomous flux function
  $f^{\epsilon_n}$.

  Consider each term separately. Since $\eta_k^\pm$ are Lipschitz
  continuous function with Lipschitz constant $1$, we can estimate the
  first term in~\eqref{eq:r1} as follows:
  \begin{align}
    \nonumber
    & \int_0^T \int_{\reali_+}
      \eta_k^\pm \left(u^{\epsilon_n} (t,x)\right)
      \, \partial_t \phi (t,x) \d{x} \d{t}
    \\
    \nonumber
    = \ &
             \int_0^T \!\int_{\reali_+}\!
             \eta_k^\pm \left(u (t,x)\right)
             \, \partial_t \phi (t,x) \d{x} \d{t}
             +
             \int_0^T \!\int_{\reali_+}\!
             \left(\eta_k^\pm \left(u^{\epsilon_n} (t,x)\right)
             -
             \eta_k^\pm \left(u (t,x)\right)\right)
             \, \partial_t \phi (t,x) \d{x} \d{t}
    \\
    \label{eq:r1ok}
    \leq \ &
             \int_0^T \!\int_{\reali_+}\!
             \eta_k^\pm \left(u (t,x)\right)
             \, \partial_t \phi (t,x) \d{x} \d{t}
             +
             \int_0^T \!\int_{\reali_+}\!
             \modulo{u^{\epsilon_n} (t,x) - u (t,x)}
             \, \partial_t \phi (t,x) \d{x} \d{t}
  \end{align}
  and the second addend in~\eqref{eq:r1ok} goes to $0$ as $\epsilon_n$
  goes to $0$.

  Concerning the second term in~\eqref{eq:r1}, proceed as follows:
  \begin{align}
    \nonumber
    & \int_0^T \int_{\reali_+}
      \Phi_{k,n}^\pm\left(u^{\epsilon_n} (t,x)\right)
      \, \partial_x \phi (t,x) \d{x} \d{t}
    \\
    \nonumber
    \leq \
    &
      \int_0^T \!\int_{\reali_+}\!
      \Phi_{k}^\pm\left(u (t,x)\right)
      \, \partial_x \phi (t,x) \d{x} \d{t}
    \\
    \nonumber
    &
      +
      \int_0^T \!\int_{\reali_+}\!
      \left(\Phi_{k,n}^\pm\left(u (t,x)\right)
      -
      \Phi_{k}^\pm\left(u (t,x)\right)\right)
      \, \partial_x \phi (t,x) \d{x} \d{t}
    \\
    \nonumber
    &
      +
      \int_0^T \!\int_{\reali_+}\!
      \left(\Phi_{k,n}^\pm\left(u^{\epsilon_n} (t,x)\right)
      -
      \Phi_{k,n}^\pm\left(u (t,x)\right)\right)
      \, \partial_x \phi (t,x) \d{x} \d{t}
    \\
    \label{eq:r1a}
    \leq \
    &
      \int_0^T \!\int_{\reali_+}\!
      \Phi_{k}^\pm\left(u (t,x)\right)
      \, \partial_x \phi (t,x) \d{x} \d{t}
    \\
    \label{eq:r1b}
    &  +
      \int_0^T \!\int_{\reali_+}\!
      \operatorname{sgn}^\pm (u-k) \left(f^{\epsilon_n}\left(u (t,x)\right)
      - f^{\epsilon_n}(k) - f\left(u (t,x)\right) + f(k) \right)
      \, \partial_x \phi (t,x) \d{x} \d{t}
    \\
    \label{eq:r1c}
    &
      +
      \norma{D f^{\epsilon_n}}_{\L\infty (\mathcal{U};\reali)}
      \int_0^T \!\int_{\reali_+}\!
      \modulo{u^{\epsilon_n} (t,x) - u(t,x)}\, \partial_x \phi (t,x) \d{x} \d{t}
  \end{align}
  and, as $\epsilon_n$ tends to $0$, \eqref{eq:r1b} goes to $0$ since
  $f^{\epsilon_n}$ converges uniformly to $f$, while~\eqref{eq:r1c}
  vanishes in the limit due to the convergence of $u^{\epsilon_n}$ to
  $u$.

  The two terms in~\eqref{eq:r2} are treated in the same way:
  \begin{align}
    \nonumber
    & \int_{\reali_+} \eta_k^\pm \left(u_o^{\epsilon_n} (x)\right)
      \, \phi (0,x) \d{x}
      -
      \int_{\reali_+} \eta_k^\pm \left(u^{\epsilon_n} (T,x)\right)
      \, \phi (T,x) \d{x}
    \\
    \label{eq:r2a}
    = \
    & \int_{\reali_+} \eta_k^\pm \left(u_o (x)\right) \, \phi (0,x) \d{x}
      -
      \int_{\reali_+} \eta_k^\pm \left(u (T,x)\right) \, \phi (T,x) \d{x}
    \\
    \label{eq:r2b}
    & +
      \int_{\reali_+}
      \left( \eta_k^\pm \left(u_o^{\epsilon_n} (x)\right)
      - \eta_k^\pm \left(u_o (x)\right) \right) \, \phi (0,x) \d{x}
    \\
    \label{eq:r2c}
    & -
      \int_{\reali_+}
      \left( \eta_k^\pm \left(u^{\epsilon_n} (T,x)\right)
      - \eta_k^\pm \left(u (T,x)\right) \right) \, \phi (T,x) \d{x}
  \end{align}
  and, since $\eta_k^\pm$ are Lipschitz continuous with constant $1$,
  \eqref{eq:r2b} and~\eqref{eq:r2c} vanish as $\epsilon_n$ goes to
  $0$, due to the assumptions~\eqref{eq:app} on the initial datum and
  to the fact that $u^{\epsilon_n}$ converges to $u$.

  Pass now to~\eqref{eq:r3}. Thanks to
  $\norma{Df^{\epsilon_n}}_{\L\infty (\mathcal{U};\reali)} \leq
  \norma{f'}_{\L\infty (\mathcal{U};\reali)}$,
  see~\eqref{eq:app}, we obtain
  \begin{align}
    \nonumber
    & \norma{Df^{\epsilon_n}}_{\L\infty (\mathcal{U};\reali)}
      \int_0^T \eta_k^\pm \left(u_b^{\epsilon_n} (t)\right) \,
      \phi (t,0) \d{t}
    \\
    \label{eq:r3a}
    \leq \
    & \norma{f'}_{\L\infty (\mathcal{U};\reali)}
      \int_0^T \!\eta_k^\pm \left(u_b (t)\right) \, \phi (t,0) \d{t}
    \\
    \label{eq:r3b}
    & +
      \norma{f'}_{\L\infty (\mathcal{U};\reali)}
      \int_0^T \! \left(
      \eta_k^\pm \left(u_b^{\epsilon_n} (t)\right)
      -
      \eta_k^\pm \left(u_b(t)\right)
      \right) \phi (t,0) \d{t}
  \end{align}
  and~\eqref{eq:r3b} vanishes as $\epsilon_n$ goes to $0$, thanks to
  the Lipschitz continuity of $\eta_k^\pm$ and the
  assumptions~\eqref{eq:app} on the boundary datum.

  Collecting together the results above, in the limit
  $\epsilon_n \to 0$, we obtain that $u$ is a weak entropy solution
  to~\eqref{eq:1aut}.

  \paragraph{A.8)~Conclusion.} Point~1.~holds by construction. For
  a.e.~$(t,x) \in [0,T] \times \reali_+$,
  $u^\epsilon (t, x) \subseteq \mathcal{U} (u_o, {u_b}_{|[0,t]})$
  and~\eqref{eq:13} imply Point~2. Formula~\eqref{eq:7} and the
  assumptions~\eqref{eq:app} on the $\epsilon$-approximation ensure
  Point~3. finally, Point~4.~follows from the inequalities
  \begin{align*}
    & \tv\left(u (t)\right) + \tv(u_b; [t,T])
    \\
    \leq \
    &   \lim_{\epsilon \to 0} \left(
             \tv\left(u^\epsilon (t)\right)
      +
      \tv\left(u_b^\epsilon (t); [t,T]\right)
             \right)
             \quad [\mbox{lower semicontinuity of the} \tv]
    \\
    \leq \ &
             \lim_{\epsilon \to 0} V^\epsilon (t)
             \quad [\mbox{see~\eqref{eq:Vepsilon}}]
    \\
    \leq \ &
             \lim_{\epsilon \to 0} V^\epsilon (0)
    \\
    \leq \ &
             \tv(u_o) + \tv (u_b)+ \modulo{u_b (0+)-u_o (0+)}.
  \end{align*}
  The above estimates ensure that
  $u \in (\L\infty \cap \BV) ([0,T] \times \reali_+; \reali)$.
\end{proofof}

\begin{proofof}{Theorem~\ref{thm:3}}
  To exploit the semigroup notation as in~\cite{AmadoriColombo97,
    BianchiniColombo, BressanLectureNotes}, we assume without loss of
  generality that $T=+\infty$.

  As in~\eqref{eq:app}, define for any positive $\epsilon$ the
  $\epsilon$--approximate fluxes
  $f^\epsilon, \, g^\epsilon \in \PLC (\reali; \reali)$, the
  $\epsilon$--approximate initial datum $u_o^\epsilon$ and boundary
  datum $u_b^\epsilon$. Let $\mathcal{D}^\epsilon$ be the set of pairs
  $\p = (u^\epsilon_o, u^\epsilon_b)$ such that
  $u^\epsilon_o \in (\L1 \cap \BV) (\reali_+; \epsilon\interi)$ and
  $u^\epsilon_b \in (\L1 \cap \BV) (\reali_+; \epsilon\interi)$,
  equipped with the norm
  $\norma{(u^\epsilon_o, u^\epsilon_b)}_{\mathcal{D}^\epsilon} =
  \max\left\{\norma{u^\epsilon_o}_{\L1 (\reali^+; \reali)}, \,
    \norma{u^\epsilon_b}_{\L1 (\reali_+;\reali)}\right\}$.
  The algorithm used in the proof of Proposition~\ref{prop:2} yields
  the semigroups
  \begin{displaymath}
    \begin{array}{@{}lcccccc}
      S^{f^\epsilon}
      & \colon
      & \reali_+
      & \times
      & \mathcal{D}^\epsilon
      & \to
      & \mathcal{D}^\epsilon
      \\
      &
      & t
      & ,
      & (u^\epsilon_o,u^\epsilon_b)
      & \mapsto
      &\left( u^\epsilon (t), \mathscr{T}_tu^\epsilon_b\right)
    \end{array}
    \quad
    \begin{array}{lcccccc@{}}
      S^{g^\epsilon}
      & \colon
      & \reali_+
      & \times
      & \mathcal{D}^\epsilon
      & \to
      & \mathcal{D}^\epsilon
      \\
      &
      & t
      & ,
      & (u^\epsilon_o,u^\epsilon_b)
      & \mapsto
      &\left( v^\epsilon (t), \mathscr{T}_tu^\epsilon_b\right)
    \end{array}
  \end{displaymath}
  using the notation~\eqref{eq:Trasl}. Note that
  $t \to u^\epsilon (t)$ and $t \to v^\epsilon (t)$ are at the same
  time $\epsilon$--approximate wave front tracking solutions
  to~\eqref{eq:14} and exact solutions to
  \begin{displaymath}
    \left\{
      \begin{array}{@{\,}l@{\qquad}r@{\,}c@{\,}l@{}}
        \partial_t u^\epsilon + \partial_x f^\epsilon (u^\epsilon) = 0
        & (t,x)
        & \in
        & \reali_+\times \reali_+
        \\
        u^\epsilon (0,x) = u_o^\epsilon (x)
        & x
        & \in
        & \reali_+
        \\
        u^\epsilon (t, 0) = u_b^\epsilon (t)
        & t
        & \in
        & \reali_+
      \end{array}
    \right.
    \mbox{ and }
    \left\{
      \begin{array}{@{\,}l@{\qquad}r@{\,}c@{\,}l@{}}
        \partial_t v^\epsilon + \partial_x g^\epsilon (v^\epsilon) = 0
        & (t,x)
        & \in
        & \reali_+\times \reali_+
        \\
        v^\epsilon (0,x) = u_o^\epsilon (x)
        & x
        & \in
        & \reali_+
        \\
        v^\epsilon (t, 0) = u_b^\epsilon (t)
        & t
        & \in
        & \reali_+ \,.
      \end{array}
    \right.
  \end{displaymath}
  Hence, applying Proposition~\ref{prop:lipdep} and using the above
  choice of the norm in $\mathcal{D}^\epsilon$, we have that
  $S_t^{f^\epsilon}$ and $S_t^{g^\epsilon}$ are Lipschitz continuous
  in both arguments, with
  \begin{equation}
    \label{eq:quellacheviene}
    \Lip (S_t^{g^\epsilon})
    \leq
    \max
    \left\{1, \norma{(g^{\epsilon})'}_{\L\infty (\mathcal{U}_t^\epsilon; \reali)}\right\}
    \leq
    \max
    \left\{1, \norma{g'}_{\L\infty (\mathcal{U}_t; \reali)}\right\} \,,
  \end{equation}
  where, using the notation~\eqref{eq:19},
  \begin{equation}
    \label{eq:27}
    \mathcal{U}_t^\epsilon = \mathcal{U} (u_o^\epsilon, {u_b^\epsilon}_{|[0,t]})
    \,,\qquad
    \mathcal{U}_t = \mathcal{U} (u_o, {u_b}_{|[0,t]})
    \quad \mbox{ and } \quad
    \mathcal{U}_t^\epsilon \subseteq \mathcal{U}_t
  \end{equation}
  due to~\eqref{eq:app}. By~\cite[Theorem~2.9]{BressanLectureNotes},
  \begin{eqnarray}
    \nonumber
    & &
        \norma{u^\epsilon (t) - v^\epsilon (t)}_{\L1 (\reali_+;  \reali)}
    \\
    \nonumber
    & = &
          \norma{
          S_t^{f^\epsilon} (u^\epsilon_o, u^\epsilon_b)
          -
          S_{t}^{g^\epsilon} (u^\epsilon_o, u^\epsilon_b)}_{\L1 (\reali_+; \reali^2)}
    \\
    \label{eq:liminf}
    & \leq &
             \Lip (S^{g^\epsilon}_t)
             \int_0^t \!
             \liminf_{h \to 0} \frac{1}{h}
             \norma{ S_h^{g^\epsilon} S_\tau^{f^\epsilon} (u^\epsilon_o, u^\epsilon_b)
             -
             S_h^{f^\epsilon} S_\tau^{f^\epsilon} (u^\epsilon_o, u^\epsilon_b)} _{\L1 (\reali_+; \reali^2)} \!
             \d\tau.
  \end{eqnarray}
  To simplify the notation, introduce
  $(w, w_b) = S^{f^\epsilon}_\tau (u^\epsilon_o,u^\epsilon_b)$.
  Outside a finite set of times $\tau$, each Riemann problem for
  $f^\epsilon$ in $w$ is solved by a single wave with speed
  $\lambda^f$. Let $\bar x$ be either $0$ or a point of jump in
  $w$. If $\bar x = 0$, set $w^\ell = w_b(0+) = u^\epsilon_b (\tau+)$,
  whereas $w^\ell = w (\bar x-)$ when $\bar x >0$. In both cases, let
  $w^r = w (\bar x+)$.  In general, the solution to the Riemann
  problem for $g^\epsilon$ with data $w^\ell$ and $w^r$ contains
  $n^\ell$ waves with speeds
  $\lambda^g_1 < \cdots < \lambda^g_{n^\ell} \leq \lambda^f$ and $n^r$
  waves with speeds
  $\lambda^f < \lambda^g_{n^\ell +1} < \cdots <
  \lambda^g_{n^\ell+n^r}$, see Figure~\ref{fig:incubo}.
  \begin{figure}[!h]
    \centering
    \includegraphics[width=.6\textwidth,trim=10 5 0 0,
    clip=true]{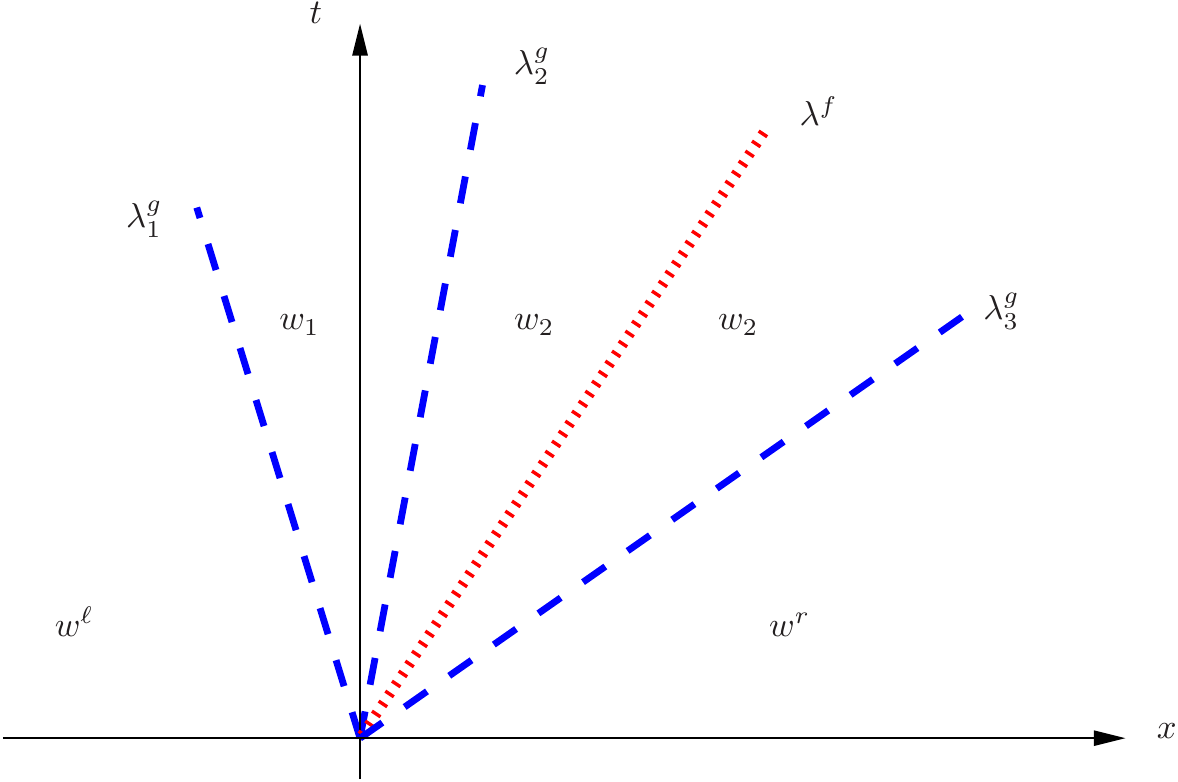}
    \caption{Notation used in the proof of Theorem~\ref{thm:3} with
      $n^\ell = 2$ and $n^r=1$.}
    \label{fig:incubo}
  \end{figure}
  Assume that the intermediate states are increasing
  $w^\ell < w_1 < \cdots < w_{n^\ell} < w_{n^\ell+1} < \ldots <
  w_{n^\ell+n^r} < w^r$,
  the other case being entirely analogous.  For a sufficiently small
  $\delta > 0$, call $I_\delta=[0,\delta]$ if $\bar x =0$ and
  $I_\delta = [\bar x-\delta, \bar x+\delta]$ if $\bar x>0$. We
  compute the integrand in~\eqref{eq:liminf} on $I_\delta$ through a
  repeated use of Rankine--Hugoniot condition:
  \begin{eqnarray}
    \nonumber
    &
    & \frac{1}{h}
      \, \norma{ S_h^{g^\epsilon} \, S_\tau^{f^\epsilon} (u^\epsilon_o, u^\epsilon_b)
      -
      S_h^{f^\epsilon} \, S_\tau^{f^\epsilon} (u^\epsilon_o, u^\epsilon_b)} _{\L1 (I_\delta; \reali^2)}
    \\
    \nonumber
    & =
    & \frac{1}{h}
      \, \norma{ S_h^{g^\epsilon}  (w, w_b)
      -
      S_h^{f^\epsilon} (w, w_b)}_{\L1 (I_\delta; \reali^2)}
    \\
    \nonumber
    & =
    & \sum_{i=1}^{n^\ell - 1}
      \modulo{w_i - w^\ell} \,  \modulo{\lambda^g_{i+1} - \lambda^g_i}
      +
      \modulo{w_{n^\ell} - w^\ell} \,  \modulo{\lambda^f - \lambda^g_{n^\ell}}
    \\
    \nonumber
    &
    & +
      \modulo{w^r - w_{n^\ell}} \,
      \modulo{\lambda^g_{n^\ell+1} - \lambda^f}
      +
      \sum_{i=1}^{n^r-1}
      \modulo{w^r - w_{n^\ell+i}} \,
      \modulo{\lambda^g_{n^\ell+i+1} - \lambda^g_{n^\ell+i}}
    \\
    \nonumber
    & =
    & \sum_{i=1}^{n^\ell - 1}
      (w_i - w^\ell) \,  (\lambda^g_{i+1} - \lambda^g_i)
      +
      (w_{n^\ell} - w^\ell) \,  (\lambda^f - \lambda^g_{n^\ell})
    \\
    \nonumber
    &
    & +
      (w^r - w_{n^\ell}) \,
      (\lambda^g_{n^\ell+1} - \lambda^f)
      +
      \sum_{i=1}^{n^r-1}
      (w^r - w_{n^\ell+i}) \,
      (\lambda^g_{n^\ell+i+1} - \lambda^g_{n^\ell+i})
    \\
    \label{eq:no}
    & =
    & \left(g^\epsilon (w^r) - g^\epsilon (w_{n^\ell})\right)
      -
      \left(g^\epsilon (w_{n^\ell}) - g^\epsilon (w^\ell)\right)
      +
      (w_{n^\ell} - w^\ell) \lambda^f
      -
      (w^r - w_{n^\ell}) \lambda^f. \qquad\
  \end{eqnarray}
  Note that by Oleinik Entropy
  condition~\cite[Formula~(8.4.3)]{DafermosBook}
  \begin{displaymath}
    \frac{f^\epsilon (w^r) - f^\epsilon (w_{n^\ell})}{w^r - w_{n^\ell}}
    \leq
    \lambda^f
    \leq
    \frac{f^\epsilon (w_{n^\ell}) - f^\epsilon (w^\ell)}{w_{n^\ell} - w^\ell} \, ,
  \end{displaymath}
  so that, using~\eqref{eq:app} and the fact that
  $w^\ell, w_{n^\ell}, w^\ell \in \epsilon\interi$, continuing the
  estimate~\eqref{eq:no}, we obtain:
  \begin{eqnarray*}
    &
    & \frac{1}{h}
        \, \norma{ S_h^{g^\epsilon} \, S_\tau^{f^\epsilon} (u^\epsilon_o, u^\epsilon_b)
        -
        S_h^{f^\epsilon} \, S_\tau^{f^\epsilon} (u^\epsilon_o, u^\epsilon_b)} _{\L1 (I_\delta; \reali^2)}
    \\
    & \leq
    & \left((g^\epsilon-f^\epsilon) (w^r) - (g^\epsilon-f^\epsilon) (w_{n^\ell})\right)
             -
             \left((g^\epsilon-f^\epsilon) (w_{n^\ell}) - (g^\epsilon-f^\epsilon) (w^\ell)\right)
    \\
    & =
    & \left((g - f) (w^r) - (g - f) (w_{n^\ell})\right)
          -
          \left((g - f) (w_{n^\ell}) - (g - f) (w^\ell)\right)
    \\
    & \leq
    & \norma{D (g-f)}_{\C0 ([w_{n^\ell}, w^r]; \reali)} \; \modulo{w^r - w_{n^\ell}}
             +
             \norma{D (g-f)}_{\C0 ([w^\ell, w_{n^\ell}]; \reali)} \; \modulo{w_{n^\ell}-w^\ell}
    \\
    & \leq
    & \norma{D (g-f)}_{\C0 ([w^\ell, w^r]; \reali)} \; \modulo{w^r - w^\ell} \,.
  \end{eqnarray*}
  By~\eqref{eq:27},
  $\mathcal{U}^\epsilon_\tau = \mathcal{U} (u^\epsilon_o,
  {u^\epsilon_b}_{|{[0,\tau]}}) \supseteq \mathcal{U} (w, w_b)$.
  Considering all Riemann problems for $u^\epsilon$ at time $\tau$
  along $\reali_+$, the integrand in~\eqref{eq:liminf} becomes
  \begin{eqnarray*}
    &
    &\frac{1}{h} \, \norma{ S_h^{g^\epsilon} \, S_\tau^{f^\epsilon}
        (u^\epsilon_o, u^\epsilon_b) - S_h^{f^\epsilon} \, S_\tau^{f^\epsilon} (u^\epsilon_o,
        u^\epsilon_b)}_{\L1 (\reali_+; \reali^2)}
    \\
    & \leq
    & \norma{D (g - f)}_{\C0 (\mathcal{U}^\epsilon_\tau; \reali)}
             \left( \tv(w) + \modulo{w_b(0+) - w (0+)} \right).
  \end{eqnarray*}
  Exploiting the functional $V^\epsilon$ defined
  in~\eqref{eq:Vepsilon} and the fact that
  $V^\epsilon (\tau) \leq V^\epsilon (0)$, we obtain
  \begin{eqnarray*}
    &
    & \frac{1}{h} \,
      \norma{
      S_h^{g^\epsilon} \, S_\tau^{f^\epsilon} (u^\epsilon_o, u^\epsilon_b)
      -
      S_h^{f^\epsilon} \, S_\tau^{f^\epsilon} (u^\epsilon_o, u^\epsilon_b)}_{\L1 (\reali_+; \reali^2)}
    \\
    & \leq
    & \norma{D (g - f)}_{\C0 (\mathcal{U}^\epsilon_\tau; \reali)}
      \left(
      \tv(u^\epsilon_o)
      +
      \tv\left(u^\epsilon_b; [0,\tau] \right)
      +
      \modulo{ u^\epsilon_b(0+) - u^\epsilon_o (0+)}
      \right).
  \end{eqnarray*}
  Hence, \eqref{eq:liminf} becomes
  \begin{equation}
    \label{eq:22}
    \begin{aligned}
      & \norma{u^\epsilon (t) - v^\epsilon (t)}_{\L1 (\reali_+;
        \reali)}
      \\
      \leq \ & t \, \Lip (S^{g^\epsilon}_t) \, \norma{D (g - f)}_{\C0
        (\mathcal{U}^\epsilon_t; \reali)} \left( \tv(u^\epsilon_o) +
        \tv\left(u^\epsilon_b; [0,t]\right) + \modulo{
          u^\epsilon_b(0+) - u^\epsilon_o (0+)} \right),
    \end{aligned}
  \end{equation}
  where $\Lip (S^{g^\epsilon}_t)$ is estimated as
  in~\eqref{eq:quellacheviene}.

  \smallskip

  Let now $u$ and $v$ be the solutions to the
  problems~\eqref{eq:14}. Similarly to above, let $\mathcal{D}$ be the
  set of pairs $\p = (u_o, u_b)$ such that
  $u_o\in (\L1 \cap \BV) (\reali_+; \reali)$ and
  $u_b\in (\L1 \cap \BV) (\reali_+; \reali)$. Thanks to
  Proposition~\ref{prop:2}, the following two semigroups are then
  defined as limit of the semigroups $S^{f^\epsilon}$ and
  $S^{g^\epsilon}$ introduced above:
  \begin{displaymath}
    \begin{array}{@{}lcccccc}
      S^{f}
      & \colon  & \reali_+  & \times & \mathcal{D}  & \to
      & \mathcal{D}
      \\
      & & t & ,& (u_o,u_b) & \mapsto
      &\left( u (t), \mathscr{T}_tu_b\right)
    \end{array}
    \quad
    \begin{array}{lcccccc@{}}
      S^{g} & \colon  & \reali_+  & \times & \mathcal{D}  & \to
      & \mathcal{D}
      \\
            & & t & ,& (u_o,u_b) & \mapsto
      &\left( v (t), \mathscr{T}_tu_b\right) \,.
    \end{array}
  \end{displaymath}
  Let $u^\epsilon_o$ and $u^\epsilon_b$ approximate $u_o$ and $u_b$ as
  in~\eqref{eq:app}. Clearly
  $\left(u^\epsilon_o, u^\epsilon_b\right) \in \mathcal{D}^\epsilon$.
  Compute
  \begin{align}
    \nonumber \norma{u(t) - v (t)}_{\L1 (\reali_+; \reali)} = \ &
                                                                  \norma{S^f_t (u_o, u_b) - S^g_t (u_o,u_b)}_{\L1 (\reali_+;
                                                                  \reali^2)}
    \\
    \label{eq:a1}
    \leq \ & \norma{S^f_t (u_o, u_b) - S^{f^\epsilon}_t
             (u^\epsilon_o,u^\epsilon_b)}_{\L1 (\reali_+; \reali^2)}
    \\
    \label{eq:a2}
                                                                & + \norma{S^{f^\epsilon}_t (u^\epsilon_o,u^\epsilon_b) -
                                                                  S^{g^\epsilon}_t (u^\epsilon_o,u^\epsilon_b)}_{\L1 (\reali_+;
                                                                  \reali^2)}
    \\
    \label{eq:a3}
                                                                & + \norma{S^{g^\epsilon}_t (u^\epsilon_o,u^\epsilon_b) - S^g_t
                                                                  (u_o,u_b)}_{\L1 (\reali_+; \reali^2)}.
  \end{align}
  Thanks to~\eqref{eq:app} and~\eqref{eq:22}, the second
  addend~\eqref{eq:a2} can be estimated as
  \begin{displaymath}
    \begin{array}{@{}c@{\;}l@{}}
      &
        \displaystyle
        \norma{
        S^{f^\epsilon}_t (u^\epsilon_o,u^\epsilon_b)
        -
        S^{g^\epsilon}_t (u^\epsilon_o,u^\epsilon_b)}_{\L1 (\reali_+; \reali^2)}
      \\
      \leq
      & \displaystyle
        t \, \Lip (S^{g^\epsilon}_t) \, \norma{D (g - f)}_{\C0 (\mathcal{U}^\epsilon_t; \reali)}
        \left(
        \tv(u^\epsilon_o)
        +
        \tv(u^\epsilon_b; [0,t])
        +
        \modulo{u^\epsilon_b(0+) - u^\epsilon_o (0+)}
        \right)
      \\
      \leq
      & \displaystyle
        t \, \max \! \left\{1, \norma{g'}_{\L\infty (\mathcal{U}_t;\reali)}\right\}
        \norma{D (g - f)}_{\C0 (\mathcal{U}_t; \reali)}
        \left[
        \tv \! (u_o)
        +
        \tv \! \left(u_b; [0,t]\right)
        +
        \modulo{ u_b(0+) - u_o(0+)}
        \right]
    \end{array}
  \end{displaymath}
  where we used~\eqref{eq:quellacheviene} and~\eqref{eq:27}. The
  terms~\eqref{eq:a1} and~\eqref{eq:a3} converge to $0$ as
  $\epsilon \to 0$, due to the construction of the
  $\epsilon$--solutions above. The proof is completed.
\end{proofof}

\subsection{Proofs related to the Non Autonomous IBVP on the
  Half--Line}
\label{subs:NAHL}

\begin{proofof}{Proposition~\ref{prop:2t}} The proof consists of
  several steps.

  \paragraph{N.1)~Construction of the Approximate Solutions.} For
  $n \in \naturali$ and $i= 0, \ldots, 2^n$, define
  $T^i_n = \frac{i}{2^n} \, T$. For $i=0, \ldots, 2^{n}-1$, we
  recursively consider the autonomous problems
  \begin{equation}
    \label{eq:21}
    \left\{
      \begin{array}{l@{\qquad}r@{\,}c@{\,}l}
        \partial_t u^i_n + \partial_x f (T^i_n,u^i_n) = 0
        &
          (t,x) & \in & [T^i_n, T^{i+1}_n] \times \reali_+
        \\
        u^i_n (T^i_n,x) = u^{i-1}_n (T^i_n, x)
        &
          x & \in & \reali_+
        \\
        u^i_n (t, 0) = u_b (t)
        &
          t & \in & [T^i_n, T^{i+1}_n] ,
      \end{array}
    \right.
  \end{equation}
  where we set $u^{-1}_n = u_o$.  Each of these problems falls within
  the scope of Proposition~\ref{prop:2}. Therefore, for any
  $\epsilon>0$ define as in~\eqref{eq:app} the $\epsilon$-approximate
  initial and boundary data $u_o^\epsilon$ and $u_b^\epsilon$.
  Moreover, for $i=0, \ldots, 2^n-1$, define the
  $\epsilon$-approximate fluxes $u \to f^\epsilon (T^i_n, u)$. Call
  $u^{i,\epsilon}_n$ the wave front tracking $\epsilon$-approximate
  solution to~\eqref{eq:21} constructed as in
  Proposition~\ref{prop:2}. Then, the solution $u^i_n$
  to~\eqref{eq:21} satisfies
  $u^i_n = \lim_{\epsilon\to 0} u^{i,\epsilon}_n$.

  For $i=0, \ldots, 2^n-1$ define
  \begin{align}
    \label{eq:uin}
    u_n (t)
    = \
    & u^i_n (t) % \quad \mbox{ if } \quad t \in [T^i_n, T^{i+1}_n]
      \quad \mbox{ for } \quad t \in [T^i_n, T^{i+1}_n]
    \\
    \label{eq:Vin}
    V^{i,\epsilon}_n
    = \
    & \tv\left(u^{i-1,\epsilon}_n (T^i_n)\right) +
      \tv\left(u_b^\epsilon; [T^i_n,T]\right) +
      \modulo{u^\epsilon_b (T^i_n+) - u^{i-1,\epsilon}_n (T^i_n, 0+)}
    \\
    \nonumber
    \mathcal{U}_t = \
    & \mathcal{U} (u_o, {u_b}_{|[0,t]})
      \quad \mbox{ and } \quad \mathcal{U} = \mathcal{U}_T
      \quad \mbox{with the notation~\eqref{eq:19}}
    \\
    \nonumber L = \
    & 1 + \norma{\partial_u f}_{\L\infty ([0,T]\times
      \mathcal{U}; \reali)}
    \\
    \nonumber K = \
    & \tv\left(u_{o} \right) + \tv\left(u_b; [0,
      T]\right) + \modulo{u_b (0+) - u_o (0+)} \,.
  \end{align}
  The quantity $V^{i,\epsilon}_n$ is the functional defined
  in~\eqref{eq:Vepsilon} computed at time $t=T^i_n$. Hence,
  by~\textbf{A.3} in the proof of Proposition~\ref{prop:2}, we
  recursively obtain
  \begin{equation}
    \label{eq:8}
    V^{i,\epsilon}_n \leq V^{i-1,\epsilon}_n
    \quad \mbox{ for all } \quad i=1, \ldots, 2^n-1 \,.
  \end{equation}
  \begin{figure}[!h]
    \centering
    \includegraphics[width=.6\textwidth,trim=0 6 0 0,
    clip=true]{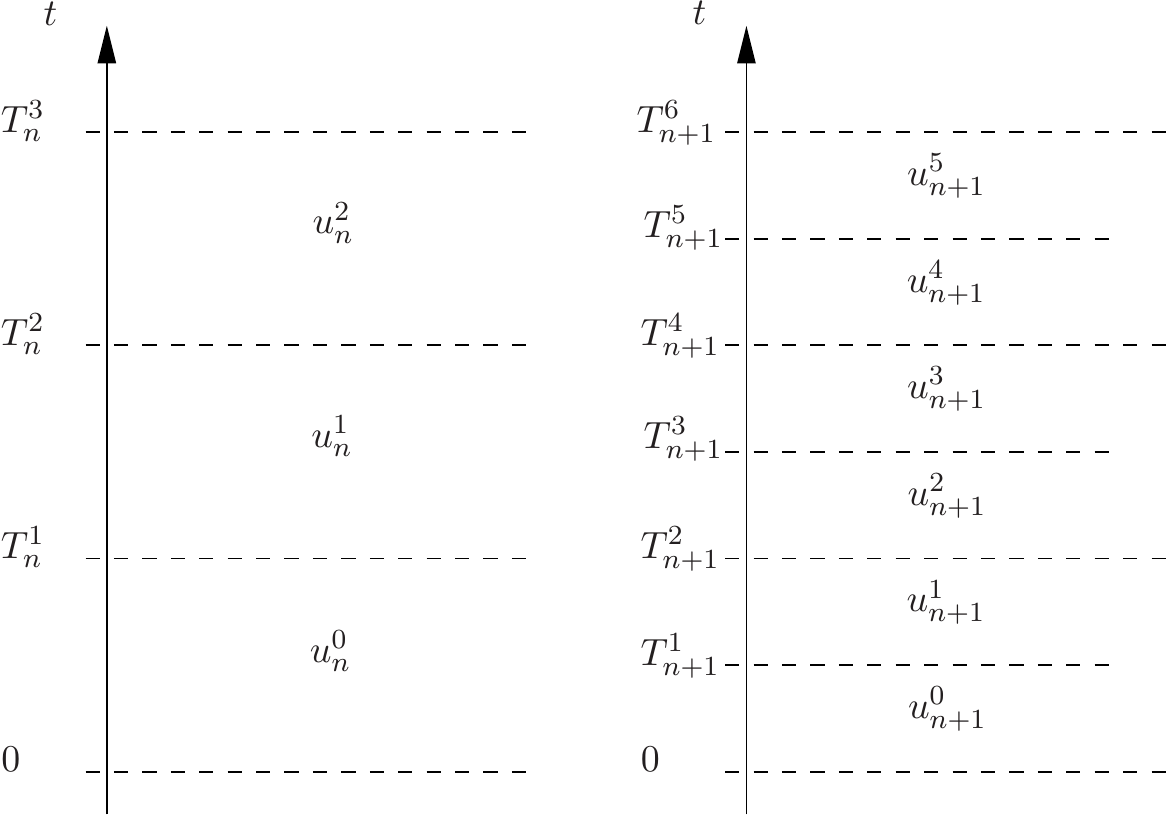}
    \caption{Relations between the time interval $[T^i_n, T^{i+1}_n]$,
      where the approximate solution is $u^i_n$, and the time
      intervals $[T^{2i}_{n+1}, T^{2i+1}_{n+1}]$ and
      $[T^{2i+1}_{n+1}, T^{2i+2}_{n+1}]$, where the approximate
      solutions are $u^{2i}_{n+1}$ and $u^{2i+1}_{n+1}$,
      see~\eqref{eq:21}.}
    \label{fig:times}
  \end{figure}

  \paragraph{N.2)~$\boldsymbol{u_n}$ is a Cauchy sequence in
    $\C0\left([0,T]; \L1 (\reali_+; \reali)\right)$.} Here and in what
  follows, we use the norm
  $\norma{u}_{\C0\left([0,T]; \L1 (\reali_+; \reali)\right)} = \sup_{t
    \in [0,T]} \norma{u (t)}_{\L1 (\reali_+; \reali)}$.
  It is sufficient to obtain
  \begin{equation}
    \label{eq:O}
    \norma{u_{n+1} - u_n}_{\C0\left([0,T]; \L1 (\reali_+; \reali)\right)} = \O \, 2^{-n}
  \end{equation}
  as soon as the constant $\O$ is independent of $n$, which in turn
  follows from the bounds
  \begin{equation}
    \label{eq:23}
    \!\!\!\!\!
    \begin{array}{@{}r@{\;}c@{\;}l@{\quad}r@{\,}c@{\,}l@{}}
      \norma{u^{2j}_{n+1} (t) - u^{j}_n (t)}_{\L1 (\reali_+; \reali)}
      &  \leq
      & j L  K
        \norma{\partial_t \partial_u f}_{\L\infty ([0,T]\times \mathcal{U}; \reali)}
        \left(\frac{T}{2^{n+1}}\right)^2
      & t
      & \in
      & [T^j_n, T^{2j+1}_{n+1}]
      \\
      \norma{u^{2j+1}_{n+1} (t) - u^{j}_n (t)}_{\L1 (\reali_+; \reali)}
      & \leq
      & (j+1) L  K
        \norma{\partial_t \partial_u f}_{\L\infty ([0,T]\times \mathcal{U}; \reali)}
        \left(\frac{T}{2^{n+1}}\right)^2
      & t
      & \in
      & [T^{2j+1}_{n+1}, T^{j+1}_n].
    \end{array}
    \!\!\!\!\!
  \end{equation}
  Fix $n$ and proceed inductively on $j$.

  \noindent$\star$ $j=0$. Assume first that $t \in [0, T^1_{n+1}]$,
  see Figure~\ref{fig:times}. By~\eqref{eq:21} we immediately have
  $u^0_{n+1} (t) = u^0_n (t)$ for $t \in [0, T^1_{n+1}]$.  Let now
  $t \in [T^1_{n+1},T^1_n]$, see Figure~\ref{fig:times}. Compute
  \begin{equation}
    \label{eq:2}
    \!\!\!\!\!\!\!\!\!\!
    \begin{array}{@{}c@{\,}l@{}}
      & \norma{u^1_{n+1} (t) - u^0_n (t)}_{\L1 (\reali_+; \reali)}
      \\
      \leq
      & \norma{u^1_{n+1} (t) - u^{1,\epsilon}_{n+1}
        (t)}_{\L1(\reali_+; \reali)}
        +
        \norma{u^{1,\epsilon}_{n+1} (t) - u^{0,\epsilon}_n (t)}_{\L1 (\reali_+; \reali)}
        +
        \norma{u^{0,\epsilon}_{n} (t) - u^0_n (t)}_{\L1 (\reali_+; \reali)}.
    \end{array}
    \!\!\!\!\!\!\!\!\!\!\!\!\!\!\!
  \end{equation}
  Focus on the term in the middle: an application of~\eqref{eq:22},
  yields
  \begin{align*}
    & \norma{u^{1,\epsilon}_{n+1} (t) - u^{0,\epsilon}_n (t)}_{\L1 (\reali_+; \reali)}
    \\
    \leq \
    & L
      \sup_{u \in \mathcal{U}}
      \modulo{\partial_u f^\epsilon (T^1_{n+1},u) - \partial_u f^\epsilon (0,u)} (t-T^1_{n+1})
    \\
    & \times \left(
      \tv\left(u^{0,\epsilon}_{n+1} (T^1_{n+1})\right)
      +
      \tv\left(u^\epsilon_b; [T^1_{n+1}, t]\right)
      +
      \modulo{u^\epsilon_b (T^1_{n+1}+) - u^{0,\epsilon}_{n+1} (T^1_{n+1}, 0+)}
      \right)
    \\
    \leq \
    & L \,
      \norma{\partial_t \partial_u f}_{\L\infty ([0,T]\times\mathcal{U};\reali)}
      \left(
      \frac{T}{2^{n+1}}\right)^2
      \left(V^{1,\epsilon}_{n+1} - \tv\left(u^\epsilon_b; [t,T]\right)
      \right)
    \\
    \leq \
    & L \,
      \norma{\partial_t \partial_u f}_{\L\infty ([0,T]\times\mathcal{U};\reali)}
      \left(
      \frac{T}{2^{n+1}}\right)^2
      \left(V^{0,\epsilon}_{n+1} - \tv\left(u^\epsilon_b; [t,T]\right)
      \right)
    \\
    \leq \
    & L \,
      \norma{\partial_t \partial_u f}_{\L\infty ([0,T]\times\mathcal{U};\reali)} \;
      \left(\frac{T}{2^{n+1}}\right)^2
      \left(
      \tv\left(u^\epsilon_{o} \right)
      +
      \tv\left(u^\epsilon_b; [0, t]\right)
      +
      \modulo{u^\epsilon_b (0+) - u^\epsilon_o (0+)}
      \right)
    \\
    \leq \
    & L \,
      \norma{\partial_t \partial_u f}_{\L\infty ([0,T]\times\mathcal{U};\reali)} \;
      \left(\frac{T}{2^{n+1}}\right)^2
      \left(
      \tv\left(u_{o} \right)
      +
      \tv\left(u_b; [0, t]\right)
      +
      \modulo{u_b (0+) - u_o (0+)}
      \right)
    \\
    \leq \
    & L \, K \,
      \norma{\partial_t \partial_u f}_{\L\infty ([0,T]\times\mathcal{U};\reali)}
      \left(\frac{T}{2^{n+1}}\right)^2,
  \end{align*}
  where we used~\eqref{eq:8} and~\eqref{eq:app}. Inserting the above
  estimate in~\eqref{eq:2} and letting $\epsilon \to 0$ yield the
  desired result.

  \smallskip

  \noindent$\star$ $j=1$. Assume first $t \in [T^1_{n},T^3_{n+1}]$,
  see Figure~\ref{fig:times}. In this time interval, the $n$- and the
  $(n+1)$-problem have the same flux, since $T^1_n=T^2_{n+1}$. An
  application of Proposition~\ref{prop:lipdep} to the autonomous
  problem~\eqref{eq:21} and using the result in the previous step
  $j=0$,
  \begin{align*}
    & \norma{u^{2}_{n+1} (t) - u^{1}_{n} (t)}_{\L1 (\reali_+; \reali)}
    \\
    \leq \
    & \norma{u^{1}_{n+1} (T^1_n) - u^{0}_n (T^1_n)}_{\L1 (\reali_+; \reali)}
    \\
    \leq  \
    & L \,
      \norma{\partial_t \partial_u f}_{\L\infty([0,T]\times\mathcal{U};\reali)}
      \left(\frac{T}{2^{n+1}}\right)^2
      \left(
      \tv\left(u_{o} \right)
      +
      \tv\left(u_b; [0, T^1_n]\right)
      +
      \modulo{u_b (0+) - u_o (0+)}
      \right)
    \\
    \leq \
    &  L \, K \,
      \norma{\partial_t \partial_u f}_{\L\infty ([0,T]\times\mathcal{U};\reali)}
      \left(\frac{T}{2^{n+1}}\right)^2 .
  \end{align*}
  Let now $t \in [T^3_{n+1},T^2_{n}]$, see Figure~\ref{fig:times}.
  Compute
  \begin{align}
    \label{eq:3}
    & \norma{u^3_{n+1} (t) - u^1_n (t)}_{\L1 (\reali_+; \reali)}
    \\
    \nonumber
    \leq \
    & \norma{u^3_{n+1} (t) - u^{3,\epsilon}_{n+1} (t)}_{\L1 (\reali_+; \reali)}
      +
      \norma{u^{3,\epsilon}_{n+1} (t) - u^{1,\epsilon}_{n} (t)}_{\L1 (\reali_+; \reali)}
      +
      \norma{u^{1,\epsilon}_{n} (t) - u^1_{n} (t)}_{\L1 (\reali_+; \reali)}.
  \end{align}
  Focus on the term in the middle: an application of
  Proposition~\ref{prop:lipdep} and of~\eqref{eq:22} yields
  \begin{align}
    \label{eq:questa}
    & \norma{u^{3,\epsilon}_{n+1} (t) - u^{1,\epsilon}_n (t)}_{\L1 (\reali_+; \reali)}
    \\
    \nonumber
    \leq \
    & \norma{u^{2,\epsilon}_{n+1} (T^3_{n+1}) - u^{1,\epsilon}_n (T^3_{n+1})}_{\L1 (\reali_+; \reali)}
    \\
    \nonumber
    & +
      L
      \sup_{u \in \mathcal{U}}
      \modulo{\partial_u f^\epsilon (T^3_{n+1},u) - \partial_u f^\epsilon (T^1_n,u)} (t-T^3_{n+1})
    \\
    \nonumber
    & \qquad
      \times
      \left(
      \tv\left(u^{2,\epsilon}_{n+1} (T^3_{n+1})\right)
      +
      \tv\left(u^\epsilon_b; [T^3_{n+1},t]\right)
      +
      \modulo{u^\epsilon_b (T^3_{n+1}+) - u^{2,\epsilon}_{n+1} (T^3_{n+1},0+)}
      \right) .
  \end{align}
  The term in the latter line above is estimated through a recursive
  use of~\eqref{eq:8}:
  \begin{align*}
    & \tv\left(u^{2,\epsilon}_{n+1} (T^3_{n+1})\right)
      +
      \tv\left(u^\epsilon_b; [T^3_{n+1},t]\right)
      +
      \modulo{u^\epsilon_b (T^3_{n+1}+) - u^{2,\epsilon}_{n+1} (T^3_{n+1},0+)}
    \\
    = \
    & V^{3,\epsilon}_{n+1} -
      \tv \left(u^\epsilon_b; [t,T]\right)
    \\
    \leq \
    & \tv(u^\epsilon_o)
      +
      \tv(u^\epsilon_b)
      +
      \modulo{u^\epsilon_b (0+) - u^\epsilon_{o} (0+)}
      -
      \tv \left(u^\epsilon_b; [t,T]\right)
    \\
    \leq \
    & \tv(u_o)
      +
      \tv\left(u_b, [0, t]\right)
      +
      \modulo{u_b (0+) - u_{o} (0+)}
    \\
    \leq \
    & K \,,
  \end{align*}
  where we exploit~\eqref{eq:app}. Hence, we recursively continue the
  estimate of~\eqref{eq:questa}:
  \begin{align*}
    & \norma{u^{3,\epsilon}_{n+1} (t) - u^{1,\epsilon}_n (t)}_{\L1 (\reali_+; \reali)}
    \\
    \leq \
    & \norma{u^{2,\epsilon}_{n+1} (T^3_{n+1}) - u^{1,\epsilon}_n (T^3_{n+1})}_{\L1 (\reali_+; \reali)}
      +
      L\, K \,
      \sup_{u \in \mathcal{U}}
      \modulo{\partial_u f (T^3_{n+1},u) - \partial_u f (T^1_n,u)} (t-T^3_{n+1})
    \\
    \leq \
    & 2 \, L \, K \,
      \norma{\partial_t \partial_u f}_{\L\infty ([0,T]\times\mathcal{U};\reali)}
      \left(\frac{T}{2^{n+1}}\right)^2 .
  \end{align*}
  Inserting the above estimate in~\eqref{eq:3} and letting
  $\epsilon \to 0$ yield the desired result.

  \smallskip

  \noindent$\star$ $j>1$. Assume first that
  $t \in [T^{j}_{n},T^{2j+1}_{n+1}]$.  An application of
  Proposition~\ref{prop:lipdep} to~\eqref{eq:21} and the inductive
  hypothesis yield
  \begin{eqnarray*}
    \norma{u^{2j}_{n+1} (t) - u^{j}_n (t)}_{\L1 (\reali_+; \reali)}
    & \leq
    & \norma{u^{2j-1}_{n+1} (T^j_n) - u^{j-1}_n (T^j_n)}_{\L1 (\reali_+; \reali)}
    \\
    & \leq
    & j L K \norma{\partial_t \partial_u f}_{\L\infty
      ([0,T]\times\mathcal{U};\reali)} \;
      \left(\frac{T}{2^{n+1}}\right)^2 .
  \end{eqnarray*}
  Let now $t \in [T^{2j+1}_{n+1},T^{j+1}_{n}]$. Compute
  \begin{align}
    \label{eq:25}
    & \norma{u^{2j+1}_{n+1} (t) - u^j_n (t)}_{\L1 (\reali_+; \reali)}
    \\
    \nonumber
    \leq \
    & \norma{u^{2j+1}_{n+1} (t) - u^{2j+1,\epsilon}_{n+1} (t)}_{\L1 (\reali_+; \reali)}
      +
      \norma{u^{2j+1,\epsilon}_{n+1} (t) - u^{j,\epsilon}_{n} (t)}_{\L1 (\reali_+; \reali)}
      +
      \norma{u^{j,\epsilon}_{n} (t) - u^j_{n} (t)}_{\L1 (\reali_+; \reali)}.
  \end{align}
  An application of Proposition~\ref{prop:lipdep} and of~\eqref{eq:22}
  to the term in the middle yields
  \begin{align}
    \label{eq:questaBis}
    & \norma{u^{2j+1,\epsilon}_{n+1} (t) - u^{j,\epsilon}_n (t)}_{\L1 (\reali_+; \reali)}
    \\
    \nonumber
    \leq \
    & \norma{u^{2j,\epsilon}_{n+1} (T^{2j+1}_{n+1}) - u^{j,\epsilon}_n (T^{2j+1}_{n+1})}_{\L1 (\reali_+; \reali)}
    \\
    \nonumber
    & +
      L
      \sup_{u \in \mathcal{U}}
      \modulo{\partial_u f^\epsilon (T^{2j+1}_{n+1},u) -
      \partial_u f^\epsilon (T^j_n,u)}
      (t-T^{2j+1}_{n+1})
    \\
    \nonumber
    & \qquad
      \times
      \left(
      \tv\left(u^{2j,\epsilon}_{n+1} (T^{2j+1}_{n+1})\right)
      +
      \tv\left(u_b^\epsilon; [T^{2j+1}_{n+1},t]\right)
      +
      \modulo{u_b^\epsilon (T^{2j+1}_{n+1}+) - u^{2j,\epsilon}_{n+1} (T^{2j+1}_{n+1},0+)}
      \right) .
  \end{align}
  The term in the latter line above can be estimated thanks
  to~\eqref{eq:8} and~\eqref{eq:app}
  \begin{align*}
    & \tv\left(u^{2j,\epsilon}_{n+1} (T^{2j+1}_{n+1})\right)
      +
      \tv\left(u^\epsilon_b; [T^{2j+1}_{n+1},t]\right)
      +
      \modulo{u^\epsilon_b (T^{2j+1}_{n+1}+) - u^{2j,\epsilon}_{n+1} (T^{2j+1}_{n+1},0+)}
    \\
    = \
    & V^{2j+1,\epsilon}_{n+1} -
      \tv \left(u^\epsilon_b; [t,T]\right)
    % \\
    % \leq \
    % & V^{2j,\epsilon}_{n+1} -
    %   \tv \left(u^\epsilon_b; [t,T]\right)
    \\
    \leq \
    & \ldots
    \\
    \leq \
    & \tv(u^\epsilon_o)
      +
      \tv(u^\epsilon_b)
      +
      \modulo{u^\epsilon_b (0+) - u^\epsilon_{o} (0+)}
      -
      \tv \left(u^\epsilon_b; [t,T]\right)
    \\
    \leq \
    & \tv(u_o)
      +
      \tv\left(u_b, [0, t]\right)
      +
      \modulo{u_b (0+) - u_{o} (0+)}
    \\
    \leq \
    & K \,.
  \end{align*}
  Hence, we continue the estimate of~\eqref{eq:questaBis}:
  \begin{align*}
    & \norma{u^{2j+1,\epsilon}_{n+1} (t) - u^{j,\epsilon}_n (t)}_{\L1 (\reali_+; \reali)}
    \\
    \nonumber
    \leq \
    & \norma{u^{2j,\epsilon}_{n+1} (T^{2j+1}_{n+1}) - u^{j,\epsilon}_n (T^{2j+1}_{n+1})}_{\L1 (\reali_+; \reali)}
      +
      L\, K \,
      \sup_{u \in \mathcal{U}} \modulo{\partial_u f^\epsilon (T^{2j+1}_{n+1},u) - \partial_u f^\epsilon (T^j_n,u)}
      (t-T^{2j+1}_{n+1})
    \\
    \leq \
    & (j+1) \, L \, K \, \norma{\partial_t \partial_u f}_{\L\infty ([0,T]\times\mathcal{U};\reali)}
      \left(\frac{T}{2^{n+1}}\right)^2,
  \end{align*}
  which inserted in~\eqref{eq:25} yields the desired result when
  passing to the limit $\epsilon \to 0$.

  \smallskip

  This proves~\eqref{eq:23}, obtaining~\eqref{eq:O} with
  $\O = \frac{1}{4} \,L \, K \, \norma{\partial_t \partial_u
    f}_{\L\infty ([0,T]\times\mathcal{U}; \reali)} \, T^2$,
  so that $u_n$ is a Cauchy sequence in
  $\C0\left([0,T];\L1 (\reali_+; \reali)\right)$: call $u$ its limit.

  \paragraph{N.3)~$\L\infty$ --Estimate.}
  Observe moreover that for any $t \in [0,T]$, Point 2.~in
  Proposition~\ref{prop:2} implies that
  $\norma{u_n^i (t)}_{\L\infty (\reali_+;\reali)} \leq \max\left\{
    \norma{u_n^{i-1} (T^i_n)}_{\L\infty (\reali_+;\reali)},
    \norma{u_b}_{\L\infty ([T^i_n,t]; \reali)} \right\}$
  for $i = 0, \ldots, 2^n-1$, and this recursively yields
  \begin{equation}
    \label{eq:nolimit}
    \norma{u_n (t)}_{\L\infty (\reali_+;\reali)}
    \leq
    \max\left\{
      \norma{u_o}_{\L\infty (\reali_+;\reali)},
      \norma{u_b}_{\L\infty ([0,t]; \reali)}
    \right\}
  \end{equation}
  which, in the limit $n\to+\infty$, gives Point~1.

  \paragraph{N.4)~$u$ is a Weak Entropy Solution to~\eqref{eq:1}.} For
  any $\phi \in \Cc1 (\reali \times \reali; \reali_+)$ and any
  $k \in \reali$, since each $u^i_n$ is a solution to~\eqref{eq:21} in
  the sense of Definition~\ref{def:solsk},
  \begin{align}
    \label{eq:a}
    0 \leq \
    & \int_0^T \int_{\reali_+}
      \eta_k^\pm \left(u_n (t,x)\right)
      \, \partial_t \phi (t,x) \d{x}\d{t}
    \\
    \label{eq:b}
    & +
      \sum_{i=0}^{2^n-1}   \int_{T^i_n}^{T^{i+1}_n} \int_{\reali_+}
      \Phi_{k}^\pm\left(T^i_n, u^i_n (t,x)\right)
      \, \partial_x \phi (t,x)  \d{x} \d{t}
    \\
    \label{eq:c}
    & +
      \int_{\reali_+} \eta_k^\pm \left(u_o (x)\right) \, \phi (0,x) \d{x}
      -
      \int_{\reali_+} \eta_k^\pm \left(u_n (T,x)\right) \, \phi (T,x) \d{x}
    \\
    \label{eq:d}
    & +
      \sum_{i=0}^{2^n-1} \norma{\partial_u f (T^i_n,\cdot)}_{\L\infty (\mathcal{U};\reali)}
      \int_{T^i_n}^{T^{i+1}_n} \eta_k^\pm \left(u_b (t)\right) \,
      \phi (t,0) \d{t},
  \end{align}
  with $\eta^\pm_k$ and $\Phi^\pm_k$ as in~\eqref{eq:sken}.  We
  compute the limit as $n \to +\infty$ of the lines above separately.

  Since $\eta^\pm_k$ are Lipschitz continuous functions with Lipschitz
  constant $1$, we obtain
  \begin{align*}
    [\eqref{eq:a}]
    \leq \
    & \int_0^T \int_{\reali_+}
      \eta_k^\pm \left(u (t,x)\right)
      \, \partial_t \phi (t,x) \d{x}\d{t}
    \\
    & +
      \int_0^T \int_{\reali_+}
      \left(
      \eta_k^\pm \left(u_n (t,x)\right) - \eta_k^\pm \left(u (t,x)\right)
      \right)
      \, \partial_t \phi (t,x) \d{x}\d{t}
    \\
    \leq \
    & \int_0^T \int_{\reali_+}
      \eta_k^\pm \left(u (t,x)\right)
      \, \partial_t \phi (t,x) \d{x}\d{t}
      +
      \int_0^T \int_{\reali_+}
      \modulo{u_n (t,x) - u (t,x)}
      \, \partial_t \phi (t,x) \d{x}\d{t},
  \end{align*}
  and in the limit $n \to +\infty$ we get
  \begin{displaymath}
    \lim_{n \to +\infty} [\eqref{eq:a}] = \int_0^T \int_{\reali_+}
    \eta_k^\pm \left(u (t,x)\right)
    \, \partial_t \phi (t,x) \d{x}\d{t}.
  \end{displaymath}
  Concerning~\eqref{eq:b}, compute
  \begin{align*}
    \Phi_{k}^\pm\left(T^i_n, u^i_n (t,x)\right)
    = \
    & \Phi_{k}^\pm\left(t, u (t,x)\right)
     +
      \left(
      \Phi_{k}^\pm\left(T^i_n, u (t,x)\right)
      -
      \Phi_{k}^\pm\left(t, u (t,x)\right)
      \right)
    \\
    & +
      \left(
      \Phi_{k}^\pm\left(T^i_n, u^i_n (t,x)\right)
      -
      \Phi_{k}^\pm\left(T^i_n, u (t,x)\right)
      \right).
  \end{align*}
  To estimate the second term above, introduce the set
  $\mathcal{U}_k = \mathcal{U} (u_o,{u_b}_{|[0,t]},k)$ and compute
  \begin{align*}
    & \Phi_{k}^\pm\left(T^i_n, u (t,x)\right)
      -
      \Phi_{k}^\pm\left(t, u (t,x)\right)
    \\
    = \
    & \sgn\nolimits^\pm \left(u (t,x) - k \right)
      \left(
      f (T^i_n, u (t,x)) - f (T^i_n, k)
      - f (t, u (t,x)) + f (t, k)
      \right)
    \\
    \leq \
    & \norma{\partial_t \partial_u f }_{\L\infty ([0,t]\times \mathcal{U}_k;\reali)}
      \,
      \modulo{u (t,x) - k} \,
      \modulo{t-T^i_n}
    \\
    \leq \
    & \norma{\partial_t \partial_u f }_{\L\infty ([0,t]\times \mathcal{U}_k;\reali)}
      \,
      \mathop{\rm diam}(\mathcal{U}_k) \; \frac{T}{2^n} \,,
  \end{align*}
  so that
  \begin{align*}
    & \sum_{i=0}^{2^n - 1}
      \int_{T^i_n}^{T^{i+1}_n} \int_{\reali_+}
      \left(
      \Phi_{k}^\pm\left(T^i_n, u (t,x)\right)
      -
      \Phi_{k}^\pm\left(t, u (t,x)\right)
      \right)
      \modulo{\partial_x \phi (t,x)}
      \d{x} \d{t}
    \\
    \leq \
    & \norma{\partial_t \partial_u f }_{\L\infty ([0,t]\times \mathcal{U}_k;\reali)}
      \,
      \mathop{\rm diam}(\mathcal{U}_k) \,
      \int_{\reali_+}   \sup_{t \in \reali^+}
      \modulo{\partial_x \phi (t,x)} \d{x} \;
      \sum_{i=0}^{2^n - 1} \left(\frac{T}{2^n}\right)^2,
  \end{align*}
  which clearly vanishes in the limit $n \to +\infty$.

  To estimate the third term, observe that the maps $\Phi^\pm_k$ are
  Lipschitz continuous, see~\cite[Lemma~3]{Kruzkov}, with Lipschitz
  constant
  $\norma{\partial_u f}_{\L\infty ([0,t] \times
    \mathcal{U}_t;\reali)}$, so that
  \begin{displaymath}
    \Phi_{k}^\pm\left(T^i_n, u^i_n (t,x)\right)
    -
    \Phi_{k}^\pm\left(T^i_n, u (t,x)\right)
    \leq
    \norma{\partial_u f}_{\L\infty ([0,t] \times \mathcal{U}_t;\reali)} \;
    \modulo{u^i_n (t,x) - u (t,x)} \,.
  \end{displaymath}
  Hence, in the limit $n \to + \infty$ we have
  \begin{displaymath}
    \lim_{n \to + \infty}[\eqref{eq:b}] = \int_0^T \int_{\reali_+}
    \Phi^\pm_k \left(t,u (t,x)\right)
    \partial_x \phi (t,x)
    \d{x}\d{t}.
  \end{displaymath}

  Pass to~\eqref{eq:c}:
  \begin{align*}
    & - \int_{\reali_+} \eta_k^\pm \left(u_n (T,x)\right) \, \phi (T,x) \d{x}
    \\
    = \
    & -
      \int_{\reali_+} \eta_k^\pm \left(u (T,x)\right) \, \phi (T,x) \d{x}
      +
      \int_{\reali_+}
      \left(
      \eta_k^\pm \left(u (T,x)\right)-\eta_k^\pm \left(u_n (T,x)\right)
      \right)
      \, \phi (T,x) \d{x}
    \\
    \leq \
    & -
      \int_{\reali_+} \eta_k^\pm \left(u (T,x)\right) \, \phi (T,x) \d{x}
      +
      \int_{\reali_+} \modulo{u (T,x) - u_n (T,x)} \, \phi (T,x) \d{x},
  \end{align*}
  and the second term vanishes as $n \to + \infty$.

  Concerning~\eqref{eq:d}, we immediately get
  \begin{displaymath}
    [\eqref{eq:d}] \leq
    \norma{\partial_u f}_{\L\infty ([0,T]\times\mathcal{U};\reali)}
    \int_0^T \eta^\pm_k \left(u_b (t)\right) \, \phi (t,0) \d{t}.
  \end{displaymath}

  We thus proved that $u$ solves~\eqref{eq:1} in the sense of
  Definition~\ref{def:solsk}.

  \paragraph{N.5)~Lipschitz continuity in time.} Consider
  $t_1,t_2 \in [0,T]$, with $t_1< t_2$.  Assume first that there
  exists $i \in \{0, \ldots, 2^n-1\}$ such that
  $t_1,t_2 \in [T^i_n,T^{i+1}_n]$. Call
  $\mathcal{U}_2 = \mathcal{U}_{t_2} = \mathcal{U}
  (u_o,{u_b}_{|[0,t_2]})$.
  Exploiting the wave front tracking approximation, compute
  \begin{align*}
    & \norma{u^i_n (t_1) - u^i_n (t_2)}_{\L1 (\reali_+;\reali)}
    \\
    \leq \
    & \norma{u^i_n (t_1) - u^{i,\epsilon}_n (t_1)}_{\L1 (\reali_+;\reali)}
      +
      \norma{u^{i,\epsilon}_n (t_1) - u^{i,\epsilon}_n (t_2)}_{\L1 (\reali_+;\reali)}
      +
      \norma{u^{i,\epsilon}_n (t_2) - u^i_n (t_2)}_{\L1 (\reali_+;\reali)}.
  \end{align*}
  The first and the third term converge to $0$ as $\epsilon \to 0$. To
  estimate the term in the middle, apply Formula~\eqref{eq:7} and
  exploit~\eqref{eq:8}:
  \begin{align*}
    & \norma{u^{i,\epsilon}_n (t_1) - u^{i,\epsilon}_n (t_2)}_{\L1 (\reali_+;\reali)}
    \\
    \leq \
    & \norma{\partial_u f^\epsilon (T^i_n)}_{\L\infty (\mathcal{U}_2;\reali)}
      (t_2-t_1)
    \\
    & \times
      \left(
      \tv \! \left(u^{i-1,\epsilon}_n (T^i_n)\right)
      +
      \tv \! \left(u^\epsilon_b; [T^i_n, t_2]\right)
      +
      \modulo{u^\epsilon_b (T^i_n+) - u^{i-1,\epsilon}_n (T^i_n,0+)}
      \right)
    \\
    \leq \
    & \norma{\partial_u f}_{\L\infty ([0,t_2] \times \mathcal{U}_2;\reali)}
      (t_2 - t_1)
      \left(V^{i,\epsilon}_n -
      \tv\left(u^\epsilon_b;[t_2,T]\right)
      \right)
    \\
    \leq \
    & \ldots
    \\
    \leq \
    & \norma{\partial_u f}_{\L\infty ([0,t_2] \times \mathcal{U}_2;\reali)}
      (t_2-t_1)
      \left(
      \tv \! \left(u^{\epsilon}_o\right)
      +
      \tv \! \left(u^\epsilon_b; [0, t_2]\right)
      +
      \modulo{u^\epsilon_b (0+) - u^\epsilon_o (0+)}
      \right)
    \\
    \leq \
    &  C \, (t_2-t_1)
  \end{align*}
  where
  \begin{equation}
    \label{eq:Cnew}
    C
    =
    \norma{\partial_u f}_{\L\infty ([0,t_2] \times
      \mathcal{U}_2;\reali)} \left( \tv\left(u_o \right) + \tv\left(u_b;
        [0, t_2]\right) + \modulo{u_b (0+) - u_o (0+)} \right)  \,.
  \end{equation}

  Assume now that there exist $i,j \in \{0, \ldots, 2^n-1\}$, with
  $i<j$, such that $t_1 \in [T^i_n,T^{i+1}_n]$ and
  $t_2 \in [T^j_n,T^{j+1}_n]$.  Therefore, exploiting the previous
  computation, we have
  \begin{align*}
    & \!\norma{u_n (t_1) - u_n (t_2)}_{\L1 (\reali_+;\reali)}
    \\
    \leq \
    & \! \norma{u^i_n (t_1) - u^i_n (T^{i+1}_n)}_{\L1 (\reali_+;\reali)}
      \!\!\!+\!\!\!
      \sum_{k=i+1}^{j-1} \!\!
      \norma{u^{k}_n (T^{k}_n) - u^{k}_n (T^{k+1}_n)}_{\L1(\reali_+;\reali)}
      \!\!\! + \!
      \norma{u_n^j (T^j_n) - u^j_n (t_2)}_{\L1 (\reali_+;\reali)}
    \\
    \leq \
    & C \, (T^{i+1}_n-t_1) + \sum_{k=i+1}^{j-1} C \,
      (T^{k+1}_n-T^k_n) + C \, (t_2-T^j_n)
    \\
    = \
    & C \, (t_2-t_1),
  \end{align*}
  with $C$ as in~\eqref{eq:Cnew}. Let now $n$ tend to $+\infty$: we
  obtain
  $\norma{u (t_1) - u (t_2)}_{\L1 (\reali_+;\reali)} \leq C \,
  (t_2-t_1)$, completing the proof of Point~2.

  \paragraph{N.6)~Total Variation Estimate.} Thanks to the lower
  semicontinuity of the total variation and to Point~4.~in
  Proposition~\ref{prop:2}, we obtain the proof of of Point~3.:
  \begin{displaymath}
    \tv \left(u(t)\right)
    \leq
    \lim_{n \to + \infty} \tv \left(u_n (t)\right)
    \leq
    \tv (u_o) + \tv (u_b;[0,t])+ \modulo{u_b (0+) - u_o (0+)}.
  \end{displaymath}
\end{proofof}

\begin{proofof}{Theorem~\ref{thm:3t}}
  Let $u_n$ and $v_n$ be defined as in~\eqref{eq:uin}, so that for
  $i =0, \ldots, 2^n-1$, $u^i_n$ and $v^i_n$ solve the autonomous
  IBVPs
  \begin{displaymath}
    \left\{
      \begin{array}{l@{\qquad}r@{\,}c@{\,}l}
        \partial_t u^i_n + \partial_x f (T^i_n,u^i_n) = 0
        & (t,x)
        & \in
        & [T^i_n,T^{i+1}_n] \times \reali_+
        \\
        u^i_n (T^i_n,x) = u^{i-1}_n (T^i_n,x)
        & x
        & \in
        & \reali_+
        \\
        u^i_n (t, 0) = u_b (t)
        & t
        & \in
        & [T^i_n,T^{i+1}_n]
      \end{array}
    \right.
  \end{displaymath}
  and
  \begin{displaymath}
    \left\{
      \begin{array}{l@{\qquad}r@{\,}c@{\,}l}
        \partial_t v^i_n + \partial_x g (T^i_n,v^i_n) = 0
        &
          (t,x) & \in & [T^i_n,T^{i+1}_n] \times \reali_+
        \\
        v^i_n (T^i_n,x) = v^{i-1}_n (T^i_n,x)
        &
          x & \in & \reali_+
        \\
        v^i_n (t, 0) = u_b (t)
        &
          t & \in & [T^i_n,T^{i+1}_n].
      \end{array}
    \right.
    \!\!
  \end{displaymath}
  As in the proof of Proposition~\ref{prop:2t}, for
  $i=0, \ldots, 2^n-1$ let $u^{i,\epsilon}_n$ and $v^{i,\epsilon}_n$
  be the corresponding wave front tracking solutions.  Observe that,
  for all $t\in [0,T] $,
  \begin{equation}
    \label{eq:miserve}
    \norma{u (t) -v (t)}_{\L1 (\reali_+;\reali)} =
    \lim_{n \to +\infty}\norma{u_n (t) -v_n (t)}_{\L1 (\reali_+;\reali)}.
  \end{equation}
  Focus on the right hand side of~\eqref{eq:miserve}. There
  exists $i \in\{0, \ldots, 2^n-1\}$ such that
  $t \in [T^i_n,T^{i+1}_n]$. Therefore,
  \begin{align}
    \nonumber
    &  \norma{u_n (t) -v_n (t)}_{\L1 (\reali_+;\reali)}
    \\
    \nonumber
    = \
    & \norma{u^i_n (t) -v^i_n (t)}_{\L1 (\reali_+;\reali)}
    \\
    \label{eq:ilsolito}
    \leq \
    & \norma{u^i_n (t) -u^{i,\epsilon}_n (t)}_{\L1 (\reali_+;\reali)}
             +
             \norma{u^{i,\epsilon}_n (t) -v^{i,\epsilon}_n (t)}_{\L1 (\reali_+;\reali)}
             +
             \norma{v^{i,\epsilon}_n (t) -v^i_n (t)}_{\L1 (\reali_+;\reali)}
  \end{align}
  The first and the third term in~\eqref{eq:ilsolito} converge to $0$
  as $\epsilon \to 0$, while an application of
  Proposition~\ref{prop:lipdep} and of Formula~\eqref{eq:22} allows to
  estimate the second term:
  \begin{align}
    \nonumber
    & \norma{u^{i,\epsilon}_n (t) -v^{i,\epsilon}_n (t)}_{\L1 (\reali_+;\reali)}
    \\
    \label{eq:29}
    \leq \
    & \norma{u^{i-1,\epsilon}_n (T^i_n) - v^{i-1,\epsilon}_n (T^i_n)}_{\L1 (\reali_+;\reali)}
    \\
    \nonumber %\label{eq:30}
    & +
      \max\left\{1,
      \norma{\partial_u g}_{\L\infty ([0,t] \times \mathcal{U}; \reali)}\right\}
      \norma{\partial_u (f-g)}_{\L\infty ([0,t] \times \mathcal{U}; \reali)}
    \\
    \label{eq:31}
    &  \times
      \left(\tv\left(v^{i-1,\epsilon}_n (T^i_n)\right) +
      \tv\left(u^\epsilon_b; [T^i_n,t]\right)
      + \modulo{u^\epsilon_b (T^i_n+) - v^{i-1,\epsilon}_n (T^i_n,0+)}
      \right) \!\left(t-T^i_n\right),
  \end{align}
  where $\mathcal{U} = \mathcal{U}(u_o, {u_b}_{|[0,t]})$ as
  in~\eqref{eq:19}, thanks to~\eqref{eq:app} and~\eqref{eq:nolimit}.
  Observe that the first term in~\eqref{eq:31} can be estimated
  by~\eqref{eq:8}:
  \begin{eqnarray}
    \nonumber
    &
    & \tv\left(v^{i-1,\epsilon}_n (T^i_n)\right) +
      \tv\left(u^\epsilon_b; [T^i_n,t]\right)
      + \modulo{u^\epsilon_b (T^i_n+) - v^{i-1,\epsilon}_n (T^i_n,0+)}
    \\
    \nonumber
    & =
    & V^{i-1,\epsilon}_n -\tv\left(u^\epsilon_b; [t,T]\right)
    \\
    \nonumber
    & \leq
    & \ldots
    \\
    \label{eq:31ok}
    &\leq
    & \tv\left(u_o \right) +
      \tv\left(u_b; [0,t]\right)
      + \modulo{u_b (0+) - u_o (0+)},
  \end{eqnarray}
  where in the last step we
  exploit~\eqref{eq:app}. Concerning~\eqref{eq:29}, we proceed
  recursively:
  \begin{align}
    \nonumber
    & [\eqref{eq:29}]
      % = \norma{u^{i-1,\epsilon}_n (T^i_n) - v^{i-1,\epsilon}_n
      % (T^i_n)}_{\L1 (\reali_+;\reali)}
    \\
    \nonumber
    \leq \
    & \norma{u^{i-2,\epsilon}_n
      (T^{i-1}_n) - v^{i-2,\epsilon}_n (T^{i-1}_n)}_{\L1 (\reali_+;\reali)}
    \\
    \nonumber
    & + \max\{1, \norma{\partial_u g}_{\L\infty ([0,t]\times
      \mathcal{U}; \reali)}\} \norma{\partial_u (f-g)}_{\L\infty
      ([0,t]\times \mathcal{U}; \reali)}
    \\
    \nonumber
    &  \times \!  \left[
      \tv\! \left(v^{i-2,\epsilon}_n (T^{i-1}_n)\right)\! +
      \tv\! \left(u^\epsilon_b; [T^{i-1}_n,T^i_n]\right)\!  + \modulo{u^\epsilon_b
      (T^{i-1}_n+) - v^{i-2,\epsilon}_n (T^{i-1}_n,0+)}
      \right]
      \!\! \left(T^i_n - T^{i-1}_n\right)
    \\
    \nonumber
    \leq \
    & \norma{u^{i-2,\epsilon}_n (T^{i-1}_n) - v^{i-2,\epsilon}_n
      (T^{i-1}_n)}_{\L1 (\reali_+;\reali)}
    \\
    \label{eq:29quasiok}
    & + \max\{1, \norma{\partial_u g}_{\L\infty ([0,t]\times
      \mathcal{U}; \reali)}\} \norma{\partial_u (f-g)}_{\L\infty
      ([0,t]\times \mathcal{U}; \reali)}
    \\
    \nonumber
    &  \times \left( \tv\left(u_o \right) + \tv\left(u_b;
      [0,t]\right) + \modulo{u_b (0+) - u_o (0+)}
      \right)\left(T^i_n-T^{i-1}_n\right).
  \end{align}
  Therefore, thanks to~\eqref{eq:31ok} and~\eqref{eq:29quasiok}, we
 obtain the estimate of~\eqref{eq:29}--\eqref{eq:31}:
  \begin{equation}
    \label{eq:uffa}
    \begin{aligned}
      \norma{u^{i,\epsilon}_n (t) -v^{i,\epsilon}_n (t)}_{\L1
        (\reali_+;\reali)}\!  \leq \ & \max\{1, \norma{\partial_u
        g}_{\L\infty ([0,t]\times \mathcal{U}; \reali)}\}
      \norma{\partial_u (f-g)}_{\L\infty ([0,t]\times \mathcal{U};
        \reali)}
      \\
      & \quad \times \left( \tv\left(u_o \right) + \tv\left(u_b;
          [0,t]\right) + \modulo{u_b (0+) - u_o (0+)} \right) t.
    \end{aligned}
  \end{equation}
  Inserting~\eqref{eq:uffa} in~\eqref{eq:ilsolito} and letting
  $\epsilon \to 0$, together with~\eqref{eq:miserve}, concludes the
  proof.
\end{proofof}

\begin{remark}
  {\rm If $T = +\infty$: the above constructions can be completed on
    any time interval $[0,T]$.  Thus, for any $T,T'$, we obtain two
    maps $u_{T}$ and $u_{T'}$ such that $u_{T'} (t) = u_T (t)$ for
    $t \in [0,\min\{T, T'\}]$, by Proposition~\ref{prop:lipdep},
    and the above procedures can be extended to $t \in \reali_+$.}
\end{remark}

%%%%%%%%%%%%%%%%%%
\medskip

\noindent\textbf{Acknowledgment:} The present work was supported by
the PRIN 2012 project \emph{Nonlinear Hyperbolic Partial Differential
  Equations, Dispersive and Transport Equations: Theoretical and
  Applicative Aspects}; by the INDAM--GNAMPA 2015 project
\emph{Balance Laws in the Modeling of Physical, Biological and
  Industrial Processes} and by the MATHTECH project funded by
CNR--INDAM.

\small{

  \bibliography{generale}

  \bibliographystyle{abbrv}

}

\end{document}